%% file: main.tex
\documentclass[a4paper,10pt]{article}
\usepackage[utf8]{inputenc}
\usepackage{amsmath,amsfonts,amssymb,amsthm}
\usepackage{enumerate}

\input tex-defs

\newtheorem{thm}{Theorem}[section]
\newtheorem{lem}[thm]{Lemma}
\newtheorem{prop}[thm]{Proposition}
\newtheorem{cor}[thm]{Corollary}
\newtheorem{rem}[thm]{Remark}

\numberwithin{equation}{section}

\title{Stability for a magnetic Schrödinger operator on a Riemann surface with boundary}

\AtEndDocument{\bigskip{\footnotesize%
  \textsc{Department of Mathematics, Uppsala University} \par  
  \textit{E-mail address}, J.~Andersson: \texttt{joelanderssonswe@gmail.com} \par
  \addvspace{\medskipamount}
  \textsc{School of Mathematics and Statistics, Sydney University} \par
  \textit{E-mail address}, L.~Tzou: \texttt{leo.tzou@gmail.com}
}}

\author{Joel Andersson \and Leo Tzou} 

\begin{document}

\maketitle

\begin{abstract}
We consider a magnetic Schrödinger operator $(\nabla^X)^*\nabla^X+q$ on a compact Riemann surface with boundary and prove a $\log\log$-type stability estimate in terms of Cauchy data for the electric potential and magnetic field under the assumption that they satisfy appropriate a priori bounds. We also give a similar stability result for the holonomy of the connection 1-form $X$.
\end{abstract}

\input{introduction.tex}

\input{cr.tex}

\input{system.tex}

\input{bdry.tex}

\input{system2.tex}

\input{holonomy.tex}

\bibliography{../../mybib}{}
\bibliographystyle{plain}

\end{document}

%% file: tex-defs.tex
\newcommand{\Z}{\mathbf{Z}}

\newcommand{\C}{\mathbf{C}}

\newcommand{\e}{\mathrm{e}}
\renewcommand{\i}{\mathrm{i}}
\newcommand{\eps}{\varepsilon}

\newcommand{\dbar}{\bar\del}

\newcommand{\del}{\partial}

\renewcommand{\d}{\mathrm{d}}

\renewcommand{\Im}{\mathrm{Im}\,}

\providecommand{\cite}[1]{[#1]}



%% file: introduction.tex
\section{Introduction}

Let $(M,g)$ be a compact Riemann surface with boundary $\del M$. We will consider a connection $\nabla^X$ on a complex line bundle over $M$, with connection 1-form $X$. By the associated \emph{connection Laplacian} we mean the differential operator
\begin{equation}\label{laplaceX}
\Delta^X:=(\nabla^{X})^*\nabla^{X}=(\d+\i X)^*(\d+\i X)=-\star(\d\star+\i X\wedge\star)(\d+\i X),
\end{equation}
where $\d$ denotes the exterior derivative, $\i=\sqrt{-1}$ and $\star$ is the Hodge star with respect to the metric $g$. In particular, when $X$ is real valued, $\Delta^X$ is often called the \emph{magnetic Laplacian} associated with the magnetic field $\d X$. We will restrict our attention to the case of real-valued $X$ in this work.

By adding a complex-valued potential $q$ we get the \emph{magnetic Schrödinger operator} associated with the couple $(X,q)$
\begin{equation}\label{SchrodingerX}
L=L_{X,q}:=\Delta^X+q.
\end{equation}
We denote by $H^s(M)$ the Sobolev space containing functions on $M$ with $s$ derivatives in $L^2(M)$. The \emph{Cauchy data space} $\mathcal{C}_L$ of $L$ is defined by
\begin{equation}\label{CauchyData}
\mathcal{C}_L:=\{(u,\nabla_\nu^{X}u)|_{\del M}\in H^{1/2}(\del M)\times H^{-1/2}(\del M); u\in H^{1}(M), Lu=0\},
\end{equation}
where $\nu$ denotes the outward pointing unit normal vector field to $\del M$ and $\nabla_\nu^{X}u:=(\nabla^{X}u)(\nu)$ is the normal derivative associated with $X$.

In this work we assume that we are working with two different magnetic Schrödinger operators $L_j:=L_{X_j,q_j}$ and their corresponding Cauchy data spaces $\mathcal{C}_j:=\mathcal{C}_{L_j}, j=1,2$. Assuming certain a priori bounds for the norms of $X_j$ and $q_j$, we illustrate that if the Cauchy data spaces are sufficiently similar, then so are the $q_j$:s and $X_j$:s respectively. The main results of this paper are:
\begin{thm}\label{mainthm1}
Suppose that $q_1, q_2$ are complex-valued functions and $X_1, X_2$ are real-valued 1-forms such that for some $p>2$,
\begin{equation}\label{main-apriori1}
\|q_j\|_{W^{1,p}(M)}\leq K, \quad\|X_j\|_{W^{2,p}(T^*M)}\leq K,\quad j=1,2.
\end{equation}
Denote by $\mathcal{C}_j:=\mathcal{C}_{L_j}$ the Cauchy data spaces as defined in \eqref{CauchyData} for the corresponding magnetic Schrödinger operators $L_j:=L_{X_j,q_j}$, as defined in \eqref{laplaceX}-\eqref{SchrodingerX}. Then if the distance $d(\mathcal{C}_1,\mathcal{C}_2)$ is small enough, there is an $\alpha>0$ such that
\begin{equation}\label{mainthmest1}
\|q_1-q_2\|_{L^2(M)}+\|\d(X_1-X_2)\|_{L^2(\Lambda^2(M))}\leq\frac{C}{\log^\alpha\log\frac{1}{d(\mathcal{C}_1,\mathcal{C}_2)}},
\end{equation}
where $C=C(K,M,\alpha)$.
\end{thm}

By interpolation it is then also quite immediate to deduce:

\begin{cor}\label{mainthm2}
If in addition to the assumptions in Theorem \ref{mainthm1} we have for some $k\geq3$
\begin{equation}\label{main-apriori2}
\|q_j\|_{H^k(M)}\leq K, \quad\|X_j\|_{H^k(T^*M)}\leq K,\quad j=1,2.
\end{equation}
Then there is an $\alpha>0$ such that
\begin{equation}\label{mainthmest2}
\|q_1-q_2\|_{H^s(M)}+\|\d(X_1-X_2)\|_{H^s(\Lambda^2(M))}\leq\frac{C}{\log^\alpha\log\frac{1}{d(\mathcal{C}_1,\mathcal{C}_2)}},
\end{equation}
where $C=C(K,M,\beta), 0\leq s<k$.
\end{cor}

These results further quantify the uniqueness results by Guillarmou and Tzou from 2011, \cite{gt-gafa}. They showed that the Cauchy data of the magnetic Schrödinger operator $L_{X,q}$ uniquely determines the potential $q$ and uniquely determines the connection $X$ up to so-called gauge isomorphism. This result was extended, by the same authors together with Albin and Uhlmann \cite{agtu2013}, to identification of coefficients (up to gauge) in elliptic systems in 2013. We have borrowed plenty of notations and conventions from these works.

The main idea in \cite{gt-gafa} is to rewrite $L_ju=0$ to $\dbar$-systems with matrix-valued potentials, and then apply the idea by Bukhgeim \cite{bu-2d}. In the identifiability case, one is able to do this due to certain orthogonality condition on the boundary which allows one to judiciously choose conjugation factors so that they agree on the boundary. In this work the orthogonality condition will only be an approximate one and we quantify the boundary conditions of the conjugation factors by this approximate orthogonality condition (see Section \ref{system2}). This process unfortunately causes the extra logarithm in the end result.

Another additional feature which we must consider is the uniformity of the estimates in the stationary phase expansion. As such we must construct slightly different phase functions than in \cite{gt-gafa} to be used in these solutions and refine several estimates from mentioned works. This work is done in Section \ref{sec:riemann}.

To see that the Cauchy data space cannot determine the couple $(X,q)$ completely, consider introducing a real-valued function $f$ (say in $H^2(M)$) whose restriction to the boundary $\del M$ is zero. Then the Cauchy data spaces associated with the magnetic Schrödinger operators $L_{X,q}$ and $L_{X+\d f,q}$ can be seen to coincide.

In \cite{gt-gafa} it is shown that the Cauchy data space determines the relative cohomology class of $X$. This is done through a number of steps. First by showing that $\mathcal{C}_1=\mathcal{C}_2$ implies that $\d(X_1-X_2)=0$ and $q_1=q_2$. If $M$ is simply connected this would imply that $X_1$ and $X_2$ differ by an exact 1-form. Furthermore, by boundary determination, $\iota_{\del M}^*(X_1-X_2)=0$, where $\iota_{\del M}^*$ denotes pullback by inclusion to the boundary. In the case of a more general Riemann surface, not necessarily simply connected, there are further obstructions so that $X_1$ and $X_2$ may no longer only differ by an exact 1-form. Namely, if $E$ is a complex line bundle over $M$ and there is a so-called unitary bundle isomorphism $F:E\to E$ that preserves the Hermitian inner product and satisfies $\nabla^{X_1}=F^*\nabla^{X_2}F$ with $F=\textrm{Id}$ on $E|_{\del M}$ it can again be seen that $\mathcal{C}_1=\mathcal{C}_2$. Having such a unitary bundle isomorphism corresponds to multiplication by a function $F$ on $M$ with the properties $|F|=1$ and $F|_{\del M}=1$. This is again equivalent with $\iota_{\del M}^*(X_1-X_2)=0$ on $\del M$, $\d(X_1-X_2)=0$ and that all periods of $X_1-X_2$ is an integer multiple of $2\pi$. I.e. for every closed loop $\gamma$ in $M$ it holds that
\begin{equation}\label{periods1}
\int_\gamma (X_1-X_2)\in 2\pi\Z.
\end{equation}
In this case one can deduce that
\begin{equation}\label{1formrep}
X_1-X_2=\d f+2\pi\sum_{k=1}^N n_k\omega_k,\quad\int_{\gamma_j}\omega_k=\delta_{jk}
\end{equation}
where $n_k\in\Z$, $\{\omega_k\}_{k=1}^N$ which is a basis for the first relative cohomology $H^1(M,\del M)$, dual to $\{\gamma_j\}_{j=1}^N$ which are some non-homotopically equivalent loops on $M$, $f$ is a function whose restriction to the boundary $\del M$ is zero and $\delta_{jk}$ is the usual Kronecker delta. 

In our framework we quantify the above statement by showing:
\begin{thm}\label{mainthm3}
There is an $\alpha>0$ such that if $d(\mathcal{C}_1,\mathcal{C}_2)$ is small enough, then for every closed loop $\gamma$ in $M$ it holds that
\begin{equation}
\inf_{k\in\Z}\left|\int_\gamma(X_1-X_2)-2\pi k\right|\leq \frac{C|\gamma|}{\log^\alpha\log\frac1{d(\mathcal{C}_1,\mathcal{C}_2)}},
\end{equation}
where $C=C(K,M,p,\alpha)$.
\end{thm} 

The above result is a consequence of Theorem \ref{mainthm1} and the Picard-Lindelöf Theorem. Some precaution is needed when defining relevant functions on $M$ since in our case $X_1-X_2$ is not necessarily closed, compare with \cite{gt-gafa}. This introduces extra complications and a suitable expression for $X_1-X_2$ is derived to quantify its distance from an exact gauge (see Section \ref{holonomy}).

We also refer to \cite{gtsurvey} for more background on the Calder\'on inverse problem for Schrödinger operators in dimension 2. In particular, \cite{gt-cp1} proves uniqueness for the usual Schrödinger operator $\Delta_g+q$ on a Riemann surface and \cite{gt-cppd} handles the corresponding partial data case. For similar results in Euclidean domains, see \cite{Novikov1987, Imanuvilov2010, Imanuvilov2012}. Related stability estimates can also be found in \cite{Santacesaria2013, Santacesaria2015}. A constructive method for the reconstructing an isotropic conductivity on a Riemann surface was first obtained in \cite{Henkin2011}.

%% file: cr.tex
\section{A primer on Riemann surfaces}\label{sec:riemann}
We aim here to give only a brief introduction to the geometry of compact Riemann surfaces and refer to the books \cite{Farkas1992, forster, jost-crs} for more comprehensive treatments.

For a closed Riemann surface $(M, g)$ we denote by $\Lambda^k(M), k=0,1,2$ the bundles of complex-valued $k$-forms on $M$. In particular, $\Lambda^0(M)$ contains functions on $M$ and $\Lambda^1( M)=\C T^*M$ is the complexified cotangent bundle on $M$, containing 1-forms. 

The Hodge star induces a natural splitting (in the sense of vector spaces) of the cotangent bundle
\[
\Lambda^1(M)=\Lambda^{1,0}(M)\oplus\Lambda^{0,1}(M),
\]
where $\Lambda^{1,0}(M)=T_{1,0}^*M:=\ker(\star+\i)$ and $\Lambda^{0,1}(M)=T_{0,1}^*M:=\ker(\star-\i)$ are eigenspaces to $\star$ corresponding to it's eigenvalues $\pm\i$. We say that a 1-form belonging to $\Lambda^{p,q}(M)$ is of type $(p,q)\in\{(1,0),(0,1)\}$. 
In holomorphic coordinates, $z=x+\i\, y$ in a chart $(U,\phi)$, the Hodge star $\star:\Lambda^k(M)\to\Lambda^{2-k}(M),$ acts according to: 
\[
\star(u\,\d z+v\,\d\bar z)=-\i\, u\,\d z+\i\, v\,\d\bar z,
\]
where 
$u,v$ are functions. So locally $\Lambda^{1,0}(M)$ is spanned by $\d z$ while $\Lambda^{0,1}(M)$ is spanned by $\d\bar z$. We have the natural projections
\begin{align*}
\pi_{1,0}:\Lambda^1(M)\to\Lambda^{1,0}(M),\quad\omega\mapsto\omega_{1,0}=\pi_{1,0}\omega,\\
\pi_{0,1}:\Lambda^1(M)\to\Lambda^{0,1}(M),\quad\omega\mapsto\omega_{0,1}=\pi_{0,1}\omega,
\end{align*}
and by definition $\star\omega_{1,0}=-\i\,\omega_{1,0}$ and $\star\omega_{0,1}=\i\,\omega_{0,1}$. 

A Hermitian inner product can be defined on the vector space $\Lambda^k(M)$ according to 
\begin{equation}\label{ip1forms}
\langle \lambda,\eta\rangle_{L^2(\Lambda^k(M))}:=\int_{M}\lambda\wedge\star\bar\eta,\quad\lambda,\eta\in\Lambda^k(M).
\end{equation}

We will often assume that our functions and forms belong to certain Sobolev spaces. Recall that for non-negative integers $k$ and $1\leq p\leq\infty$, a function $f$ is said to belong to the Sobolev space $W^{k,p}(M)$ if it is $k$ times weakly differentiable and all partial derivatives up to order $k$ belong to $L^p(M)$. For $k=2$ we use the common notation $H^k(M):=W^{k,2}(M)$, or $H^s(M)$ when considering non-integers $s=k$, and for $k=0$ we have $L^p(M)=W^{0,p}(M)$. For a more thorough discussion on Sobolev spaces we refer to e.g.\ Section 1.3 of \cite{Schwarz1995}, that also covers the spaces $W^{k,p}(\Lambda^l(M))$ that we also consider in the cases $l=1,2$.

\subsection{Cauchy-Riemann operators on $M$}\label{subseq:cr}
Next we will introduce the Cauchy Riemann operators $\del,\bar\del$ as mappings of $k$-forms to $(k+1)$-forms, $k=0,1$. 

On functions their action is defined by
\[
\del f := \pi_{1,0}\d f,\quad\bar\del f := \pi_{0,1}\d f.
\]
In holomorphic coordinates this is simply
\[
\del f := \del_z f\,\d z,\quad\dbar f := \del_{\bar z}f\,\d\bar z,
\]
where 
\[
\del_z=\frac12\left(\frac{\del}{\del x}-\i\,\frac{\del}{\del y}\right),\quad\del_{\bar z}=\frac12\left(\frac{\del}{\del x}+\i\,\frac{\del}{\del y}\right)
\]
are the usual Wirtinger derivatives. The space of holomorphic functions over $M$ is denoted by $\mathcal{H}(M)$ and consists of functions $f$ that satisfy $\dbar f=0$.

On 1-forms we define
\[
\del\omega:=\d\pi_{0,1}\omega=\d\omega_{0,1},\quad\dbar\omega:=\d\pi_{1,0}\omega=\d\omega_{1,0}.
\]
In coordinates this is
\[
\del(u\,\d z+v\,\d\bar z)=\del v\wedge\d\bar z,\quad\dbar(u\,\d z+v\,\d\bar z)=\dbar u\wedge\d z.
\]
It is clear that $\d=\del+\dbar$ holds for both functions and 1-forms. 
The adjoints of $\del$ and $\dbar$ are simply given by $\del^*=\i\star\dbar$ and $\dbar^*=-\i\star\del$ respectively. 
We can now define the Laplacian of a function $f$ on $M$ by
\[
\Delta f:=-2\i\star\dbar\del f=\delta\d f
\]
where $\delta=\d^*$ is the codifferential, i.e. the adjoint of $\d$ with respect to the inner product, and the $\star$ is the induced Hodge star that maps 2-forms to 0-forms.

\subsection{Existence of suitable phase functions on $M$}
We will now construct functions that are holomorphic and Morse, with uniformly bounded from below Hessian outside a neighborhood of their stationary points. These functions will be used as phase functions in latter arguments that allow us to estimate the degeneracy near stationary points. This must be taken into account when deriving correct remainder estimates in e.g.\ stationary phase expansions. 

Start out with any $\Phi\in\mathcal{H}(M)$ which is Morse, meaning that if $\hat p\in M$ is a stationary point, $\del\Phi(\hat p)=0$, then the Hessian $\del_p^2\Phi(\hat p)\neq0$. For the construction of such functions see \cite{gt-cp1}. Suppose that $\{p_1,\dots,p_n\}$ is the set of stationary points of $\Phi$ and for $k=1,\dots,n$ denote by $p_{j,k}, j=1, \dots,m_k$ the set of points for which the values $\Phi(p_{j,k})$ coincide with the critical values $\Phi(p_k)$. Then we have the following lemma:

\begin{lem}\label{l:qPhip} Let $\Phi$ be a holomorphic Morse function on a compact Riemann surface $M$ with boundary, having critical points $\{p_k\}_{k=1}^n\subset M$, and let 
\[
P_{cv}=\{p_{j,k}\in M;\Phi(p_{j,k})=\Phi(p_k), j=1,\dots,m_k,k=1,\dots,n\}
\]
be the set of points where the values of $\Phi$ coincides with a critical value. For any $\delta>0$ we let $N_\delta$ be a neighborhood of $P_{cv}$ defined by:
\[
N_\delta=\bigcup_{j,k} B_{j,k,\delta},\quad B_{j,k,\delta}=\phi^{-1}(B(\phi(p_{j,k}),\delta)),
\]
where $(U,\phi)$ are charts such that $\phi:U\to\C$ and 
\[
B(\phi(p_{j,k}),\delta)=\{z\in\C;|z-z_{j,k}|<\delta, z_{j,k}=\phi(p_{j,k})\}\subset\phi(U).
\]
Then for any $\hat p\in M\setminus N_\delta$ there exists a holomorphic Morse function $\Phi_{\hat p}$ with a critical point at $\hat p$ such that at all critical points $p$ of $\Phi_{\hat p}$ there is a $c>0$ such that 
\[
|\del_p^2\Phi_{\hat p}(p)|\geq c\delta^4,
\]
uniformly in $\hat p\in M\setminus N_\delta$.
\end{lem}

\begin{proof}
The sought after function will be defined by
\[
\Phi_{\hat p}(p):=(\Phi(p)-\Phi(\hat p))^2, \quad \hat p\in M\setminus N_\delta.
\]
Then
\[
\del\Phi_{\hat p}(p)=2(\Phi(p)-\Phi(\hat p))\del\Phi(p),
\]
so clearly $\del\Phi_{\hat p}(\hat p)=0$, and we will show that $\Phi_{\hat p}$ is Morse. Observe that, in a coordinate system $(U, \phi)$ containing $p = \phi(z)$
\[
\del^2_z\Phi_{\hat p}(\phi^{-1}(z))=2(\del_z\Phi(\phi^{-1}(z)))^2+2(\Phi(\phi^{-1}(z))-\Phi(\hat p))\del_z^2\Phi(\phi^{-1}(z)).
\] 
Suppose that $p\in M$ is a point such that
\[
0=\del\Phi_{\hat p}(p)=2(\Phi(p)-\Phi(\hat p))\del\Phi(p).
\]
Then, since $\hat p$ is bounded away from critical values of $\Phi$, we have the following two mutually exclusive cases 
\begin{enumerate}
\item $\del\Phi(p)=0$ in which case $p=p_k$ for some $k\in\{1,\dots,n\}$ or,
\item $\Phi(p)=\Phi(\hat p)$.
\end{enumerate} 
In case 1, we have with $z_k=\phi(p_k)$,
\[
\del^2_z\Phi_{\hat p}(\phi^{-1}(z_k))=2(\Phi(\phi^{-1}(z_k))-\Phi(\phi^{-1}(\hat z)))\del_z^2\Phi(\phi^{-1}(z_k)),
\] 
and since $\Phi$ is Morse there is a holomorphic chart in which $\del^2_z\Phi(\phi^{-1}(z_k))=a_k\neq0$. Also since $\hat p\notin N_\delta, \Phi(p_k)-\Phi(\hat p)\neq0$ and furthermore (for small enough $\delta>0$ and $\hat p$ close enough to $p_k$) we have in local coordinates, 
\[
|\Phi(\phi^{-1}(z_k))-\Phi(\phi^{-1}(\hat z))|=|a_k||z_k-\hat z|^2+O(|z_k-\hat z|^3).
\]
(Observe that $|\Phi(p_k)-\Phi(\hat p)|\to0$ only if $\hat p\to p_{j,k}$ for some $p_{j,k}\in P_{cv}$, and thus $|\Phi(p_k)-\Phi(\hat p)|=|\Phi(p_k)-\Phi(p_{j,k})+\Phi(p_{j,k})-\Phi(\hat p)|=|\Phi(p_{j,k})-\Phi(\hat p)|$.)
Since we only have finitely many critical points we can thus choose an $a>0$, such that in a coordinate where $p= \phi(z)$
\[
|\del_z^2\Phi_{\hat p}(\phi^{-1}(z))|>a\delta^2,\quad \text{for all }\ \hat p\in M\setminus N_\delta, \text{ for } p \in M: \del\Phi(p)=0.
\]
In case 2, when $\Phi(p)=\Phi(\hat p)$ it follows that $p\neq p_k$ for any $k$, since $\hat p\notin N_\delta$. Furthermore,
\[
\del^2_z\Phi_{\hat p}(\phi^{-1}(z))=2(\del_z\Phi(\phi^{-1}(z)))^2.
\]
Clearly the right hand side approaches 0 only when $p\to p_k$ for some $k\in\{1,\dots,n\}$ and we have by a similar argument as in case 1 that
\begin{multline*}
|z-z_k|\geq\frac{1}{c}|\Phi(\phi^{-1}(z))-\Phi(\phi^{-1}(z_k))|=\frac{1}{c}|\Phi(\phi^{-1}(\hat z))-\Phi(\phi^{-1}(z_k))|\\=\frac{|a_k||\hat z-z_k|^2+O(|\hat z-z_k|^3)}{c},
\end{multline*}
where $c>0$ is the Lipschitz constant for $\Phi\circ\phi^{-1}$. So the lower bound $|\hat z-z_k|\geq\delta$ will yield a lower bound on $|z-z_k|\gtrsim\delta^2$, hence near $p_k$ we have (in local coordinates)
\[
|\del_z\Phi(\phi^{-1}(z))|=2|a_k||z-z_k|+O(|z-z_k|^2)\geq\sqrt{b}|\hat z-z_k|^2=\sqrt b \delta^2.
\]
Thus for $p$ near $p_k$ we can again choose $b>0$ such that
\[
|\del^2_z\Phi_{\hat p}(\phi^{-1}(z))|>b\delta^4,\quad \text{for all }\ \hat p\in M\setminus N_\delta, \text{ for } p\in M: \Phi(p)=\Phi(\hat p),
\]
 and this holds uniformly in $\hat p\in M\setminus N_\delta$.
\end{proof}

\begin{rem} When the Riemann surface M is of a particularly simple type, e.g.\ if it is a equipped with a global holomorphic coordinate $z$ (such as when $M$ is just a domain in $\C$) we do not need the construction in the above lemma. Since in the latter case it holds that for every $\hat z=\phi(\hat p)$, $\Phi(z)=(z-\hat z)^2$ is a Morse holomorphic function with a critical point at $\hat z$.
\end{rem}

We are going to use the method of stationary phase with the above constructed phase function in order to later derive estimates. For convenience we replace $\delta$ by $\sqrt{\delta}$ in the above lemma, i.e.\ we consider the phase function $\Phi=\Phi_{\hat p}$ on $M\setminus N_{\sqrt\delta}$, where $|\del_p^2\Phi_{\hat p}|\geq c\delta^2$. 
 Recall that $C_0^\infty$ as usual denotes smooth and compactly supported functions/form. The dependence on the parameters $h$ and $\delta$ will be important.
\begin{lem}[Stationary phase]\label{l:stph}
Suppose $\psi$ is a smooth function and $K\subset\C$ is a set containing only one critical point $\hat z$ of $\psi$ and that $|\del_z^2\psi(\hat z)|\geq \delta^2>0$. Then for every $u\in C_0^\infty(K)$,
\begin{enumerate}
\item
\[
\left|\int u(z)\e^{2\i\psi(z)/h}\,\d z\,\d\bar z\right|\leq\frac{Ch}{\delta}\|u\|_{W^{2,\infty}(K)},\quad\text{and}
\]
\item
\[
\left|\int u(z)\e^{2\i\psi(z)/h}\,\d z\,\d\bar z-\frac{ch}{\delta}u(\hat z) \right| \leq\frac{Ch^2}{\delta^{7}}\|u\|_{W^{4,\infty}(K)}.
\]
\end{enumerate}
The constants $C>0, c>0$ are uniformly bounded with respect to $\hat z$.
\end{lem}
\begin{proof}
This follows from \cite{LHLPDO1}, Theorem 7.7.5.
\end{proof}

\subsection{The inverses of $\del$ and $\dbar$}\label{subseq:crinv}
Section 2 in \cite{gt-gafa} contain lemmas regarding the construction and boundedness of right inverses of the Cauchy-Riemann operators. These results ensures that the constructed solutions to the Dirac-systems that are considered in the paper are well-behaved. As we will use a very similar approach in Section \ref{sec:systems}, we recite some of these essential lemmas below.

\begin{lem}[Right inverse to $\dbar$, \cite{gt-gafa}]\label{l:dbarinv}
There exists an operator
\[
\dbar^{-1}:C_0^{\infty}(\Lambda^{1,0}(M))\to C^{\infty}(M)
\]
such that
\begin{enumerate}
\renewcommand{\labelenumi}{(\roman{enumi})}

\item $\dbar\dbar^{-1}\omega=\omega$ for all $\omega\in C_0^{\infty}(\Lambda^{1,0}(M))$.

\item If $\chi_j\in C_0^{\infty}(M)$ are supported in complex charts $U_j$, bi-holomorphic to a bounded open set $\Omega\subset\C$ with complex coordinate $z$, and such that $\chi=\sum\chi_j$ is equal to 1 on $M$, then as operators
\begin{equation}\label{dbarinv}
\dbar^{-1}\chi=\sum\hat\chi_j\bar T\chi_j+K,
\end{equation}
where $\hat\chi_j\in C_0^{\infty}(U_j)$ are such that $\hat\chi_j\chi_j=\chi_j$, $K$ has smooth kernel on $M\times M$ and $\bar T$ is given in the complex coordinate $z\in U_i$ by

\begin{equation}\label{cauchyop1} \bar T(f\,\d\bar z)=\frac1\pi\int_\C\frac{f(\zeta)}{z-\zeta}\,\d\zeta\wedge\d\bar\zeta.\end{equation}

\item $\dbar^{-1}$ is bounded from $L^p(\Lambda^{1,0}(M))$ to $W^{1,p}(M)$ for all $p>1$.

\end{enumerate}
\end{lem}
So by $(ii)$, the inverse can be expressed by the usual Cauchy operator, which is called $\bar T$ here, plus a smoothing term, in local coordinates. A similar result for $\dbar^{*}$ is given by

\begin{lem}[Right inverse to $\dbar^{*}$, \cite{gt-gafa}]\label{l:dbar*inv}
Let $\dbar^{*}:W^{1,p}(\Lambda^{0,1}(M))\to L^{p}(M)$, then there exists an operator
\[
\dbar^{*-1}:C_0^{\infty}(M)\to C^{\infty}(\Lambda^{0,1}(M))
\]
such that
\begin{enumerate}
\renewcommand{\labelenumi}{(\roman{enumi})}

\item $\dbar^{*}\dbar^{*-1}\varphi=\varphi$ for all $\varphi\in C_0^{\infty}(M)$.

\item If $\chi_j\in C_0^{\infty}(M)$ are supported in complex charts $U_j$, bi-holomorphic to a bounded open set $\Omega\subset\C$ with complex coordinate $z$, and such that $\chi=\sum\chi_j$ is equal to 1 on $M$, then as operators
\begin{equation}\label{dbarinv}
\dbar^{*-1}\chi=\sum\hat\chi_j T\chi_j+K,
\end{equation}
where $\hat\chi_j\in C_0^{\infty}(U_j)$ are such that $\hat\chi_j\chi_j=\chi_j$, $K$ has smooth kernel on $M\times M$ and $T$ is given in the complex coordinate $z\in U_i$ by
\begin{equation}\label{cauchyop2}
Tf(z)=\left(\frac1\pi\int_\C\frac{f(\zeta)}{\bar z-\bar\zeta}\,\d\zeta\wedge\d\bar\zeta\right)\d\bar z.
\end{equation}

\item $\dbar^{*-1}$ is bounded from $L^p(M)$ to $W^{1,p}(\Lambda^{0,1}(M))$ for all $p>1$.

\end{enumerate}
\end{lem}
Here we again observe that $(ii)$ says that the inverse can be expressed as a Cauchy-type operator plus a smoothing term, in local coordinates. For the proof we again refer to \cite{gt-gafa}. 
The main use of Lemma \ref{l:dbarinv} and \ref{l:dbar*inv} in \cite{gt-gafa} is to prove Lemma \ref{l:dbarpsiinv} and \ref{l:dbarpsi*inv}. We could also make use of Lemma \ref{l:dbarinv} and \ref{l:dbar*inv} in order to prove estimates for the solutions of the $\dbar$-systems we will consider in later sections.  In order to prove Theorem \ref{mainthm1} we will require more refined and explicit estimates than what is needed in order to prove identifiability of the pair $(X,q)$, so this is another reason for restating also the above lemmas.

Let us now assume that $M$ is strictly contained in some larger surface $N$. Suppose $p,q\in[1,\infty]$ and define the continuous extension operator from $M$ to $N$ by
\[
E:W^{k,p}(\Lambda^{0,1}(M))\to W_c^{k,p}(\Lambda^{0,1}(N)),
\]
where $W_c^{k,p}(\Lambda^{0,1}(N))$ denotes the subspace of compactly supported type $(0,1)$-forms in $W^{k,p}(\Lambda^{0,1}(N)), k=1,2$, with a range made of type $(0,1)$-forms with support in $N_\delta=\{n\in N;d(n,\bar M)\leq\delta\}$ for some $\delta>0$. We also denote by
\[
R:L^{q}(N)\to L^{q}(M)
\]
the restriction map from $N$ to $M$.

\begin{lem}[Lemma 2.2 from \cite{gt-gafa}]\label{l:dbarpsiinv}
Let $\psi$ be a real-valued smooth Morse function on $N$ and let $\dbar_\psi^{-1}:=R\dbar^{-1}\e^{-2\i\psi/h}E$. For $q>1, p>2$ there exists $C>0$ that is independent of $h$ such that for all $\omega\in W^{1,p}(\Lambda^{0,1}(M))$,
\begin{align}
\|\dbar_\psi^{-1}\omega\|_{L^q(M)}&\leq Ch^{2/3}\|\omega\|_{W^{1,p}(\Lambda^{0,1}(M))},\quad 1\leq q<2,\label{ineq:q1}\\
\|\dbar_\psi^{-1}\omega\|_{L^q(M)}&\leq Ch^{1/q}\|\omega\|_{W^{1,p}(\Lambda^{0,1}(M))},\quad 2\leq q\leq p.\label{ineq:q2}
\end{align}
By interpolation, there is thus an $\eps>0$ and $C>0$ such that for all $\omega\in W^{1,p}(\Lambda^{0,1}(M))$
\begin{equation}\label{L2est1}
\|\dbar_\psi^{-1}\omega\|_{L^2(M)}\leq Ch^{1/2+\eps}\|\omega\|_{W^{1,p}(\Lambda^{0,1}(M))}.
\end{equation}
\end{lem}

As we will be required to use the special phase functions constructed in Lemma \ref{l:qPhip} we state below a more explicit version of the above lemma in order to see the $\delta$-dependence.

\begin{lem}[Refinement of Lemma 2.2 from \cite{gt-gafa}]\label{l:dbpi}
Let $\Phi_{\hat p}$ be a holomorphic Morse function as in Lemma \ref{l:qPhip} and let $\psi=\Im\Phi_{\hat p}$. Define $\dbar_\psi^{-1}:=R\dbar^{-1}\e^{-2\i\psi/h}E$ where $\dbar^{-1}$ is the right inverse of $\dbar:W^{1,p}(M)\to L^p(\Lambda^{0,1}(M))$. Let $p>2$, then there are constants $\eps>0, C>0$ independent of $h$ and $\hat p$ such that for all $\omega\in W^{1,p}(\Lambda^{0,1}(M))$,
\[
\|\dbar_{\psi}^{-1}\omega\|_{L^2(M)}\leq\frac{Ch^{1/2+\eps}}{\delta^4}\|\omega\|_{W^{1,p}(\Lambda^{0,1}(M))}.
\]
\end{lem}

\begin{proof}
The proof is identical to that of Lemma 2.2 in \cite{gt-gafa}, with the addition that Lemma \ref{l:qPhip} is applied when doing estimates of 
\[
\frac{1}{|\dbar\psi|}\leq\frac{c}{\delta^2|z-\hat z|}
\]
near a critical point $z_0$. One observes that there will be no terms in which the $\delta$-dependence is worse than $\delta^{-4}$ (using also part 1 of Lemma \ref{l:stph} when estimating the smoothing part of $\dbar^{-1}$).
\end{proof}

The final lemma of this section follows Lemma \ref{l:dbarpsiinv}, and is proved in exactly the same way, but for a corresponding $\dbar_\psi^{*-1}$. Here the restriction and extension operators must be interpreted in a different way, namely that $R'$ restricts sections of $\Lambda^{0,1}(N)$ to $M$ and $E'$ is a continuous extension from $W^{k,p}(M)$ to $W_c^{k,p}(N), k=0,1$ where the functions in its range are supported in some $N_\delta$.

\begin{lem}[Lemma 2.3 in \cite{gt-gafa}]\label{l:dbarpsi*inv}
Let $\psi$ be a real-valued smooth Morse function on $N$ and let $\dbar_\psi^{*-1}:=R'\dbar^{*-1}\e^{2\i\psi/h}E'$. For $q>1, p>2$ there exists $C>0$ that is independent of $h$ such that for all $v\in W^{1,p}(M)$,
\begin{align}
\|\dbar_\psi^{*-1}v\|_{L^q(\Lambda^{0,1}(M))}&\leq Ch^{2/3}\|v\|_{W^{1,p}(M)},\quad 1\leq q<2,\label{ineq:q3}\\
\|\dbar_\psi^{*-1}v\|_{L^q(\Lambda^{0,1}(M))}&\leq Ch^{1/q}\|v\|_{W^{1,p}(M)},\quad 2\leq q\leq p.\label{ineq:q4}
\end{align}
By interpolation, there is thus an $\eps>0$ and $C>0$ such that for all $v\in W^{1,p}(M)$
\begin{equation}\label{L2est4}
\|\dbar_\psi^{*-1}v\|_{L^2(\Lambda^{0,1}(M))}\leq Ch^{1/2+\eps}\|v\|_{W^{1,p}(M)}.
\end{equation}
\end{lem}

The proofs of Lemma \ref{l:dbarpsiinv} and \ref{l:dbarpsi*inv} are rather long but have the advantage that they can be quite easily adapted to give further useful estimates, such as
\begin{equation}\label{L2est2}
\|E^{*}(\dbar^{-1})^{*}R^{*}(\e^{-2\i\psi/h}v)\|_{L^{2}(\Lambda^{0,1}(M))}\leq Ch^{1/2+\eps}\|v\|_{W^{1,p}(M)},
\end{equation}
for $v\in W^{1,p}(M)$ and
\begin{equation}\label{L2est3}
\|E'^{*}(\dbar^{*-1})^{*}R'^{*}(\e^{2\i\psi/h}\omega)\|_{L^{2}(M)}\leq Ch^{1/2+\eps}\|\omega\|_{W^{1,p}(\Lambda^{0,1}(M))},
\end{equation}
where it must also be assumed that $\omega|_{\del M}=0$.

The corresponding extension of \eqref{L2est4} with explicit $\delta$-dependence is given in the next lemma, with proof analogous to the one of Lemma \ref{l:dbpi}

\begin{lem}[Refinement of Lemma 2.3 from \cite{gt-gafa}]\label{l:dbp*i}
Let $\Phi_{\hat p}$ be a holomorphic Morse function as in Lemma \ref{l:qPhip} and let $\psi=\Im\Phi_{\hat p}$. Define $\dbar_\psi^{*-1}:=R'\dbar^{*-1}\e^{2\i\psi/h}E'$ where $\dbar^{*-1}$ is the right inverse of $\dbar^*:W^{1,p}(\Lambda^{0,1}(M))\to L^p(M)$. Let $p>2$, then there are constants $\eps>0, C>0$ independent of $h$ and $\hat p$ such that for all $\omega\in W^{1,p}(\Lambda^{0,1}M)$,
\[
\|\dbar_{\psi}^{*-1}v\|_{L^2(\Lambda^{0,1}(M))}\leq\frac{Ch^{1/2+\eps}}{\delta^4}\|v\|_{W^{1,p}(M)}.
\]
\end{lem}

%% file: system.tex
\section{Inverse boundary problems for systems}\label{sec:systems}

Our approach mimic the idea of Bukhgeim \cite{bu-2d} that makes it possible to study a first order differential operator represented by a matrix instead of \eqref{SchrodingerX}.

The idea is to consider the bundle $\Sigma(M):=\Lambda^0(M)\oplus\Lambda^{0,1}(M)$ and the $\dbar$-system
\begin{equation}\label{dbarsystem}
(D+V)U=0,
\end{equation}
where
\[
D=\left[\begin{array}{cc} 0 & \dbar^*\\\dbar & 0\end{array}\right],\quad V=\left[\begin{array}{cc} Q^+ & A'\\A & Q^-\end{array}\right],\quad U=\left(\begin{array}{c} u\\\omega_{0,1}\end{array}\right)\in\Sigma(M).
\]
The operator $D$ is often called a \emph{Dirac operator} and is formally self-adjoint. The potential $V$ will be built up by functions $Q^\pm$ on the diagonal and 1-forms $A,A'$ on the antidiagonal. The action of $V$ on $\Sigma(M)$ must be interpreted in the correct way, e.g. in our case it will be of the form
\[
V=\left[\begin{array}{cc} Q & \star(\bar A\wedge\cdot)\\ A & -1\end{array}\right],\quad A\in\Lambda^{0,1}(M).
\]
Hence
\[
VU=\left(\begin{array}{c} Qu+\star\bar A\wedge\omega_{0,1}\\uA-\omega_{0,1}\end{array}\right).
\]

In \cite{gt-gafa} it is assumed that $V$ is a diagonal endomorphism of $\Sigma(M)$. This condition was relaxed in \cite{agtu2013} and we will mainly follow the methodology of this work in this section.


The inner products on $\Lambda^0(M)$ and $\Lambda^{1}(M)$ induce a natural inner product on $\Sigma(M)$ as
\begin{equation}\label{ipsigma}
\langle\cdot,\cdot\rangle_{L^2(\Sigma(M))}:=\langle\cdot,\cdot\rangle_{L^2(M)}+\langle\cdot,\cdot\rangle_{L^2(\Lambda^{1}(M))}
\end{equation}
where the two inner products in the right hand side are defined by 
\eqref{ip1forms}.

We will make use of the following boundary integral identity, which is easily proved by integration by parts.
\begin{lem}\label{boundaryintid}
Suppose that $U'$ is a solution to the system $(D+V^*)U'=0$ in $M$. Then
\[
\langle (D+V)U,U'\rangle_{L^2(\Sigma(M))}=\langle U,U'\rangle_{\del M},
\]
where
\[
\langle U,U'\rangle_{\del M}:=\int_{\del M}\iota_{\del M}^*\left(u\star\overline{\omega_{0,1}'}-\star\omega_{0,1}\overline{u'}\right).
\]
In the above, $\iota_{\del M}^*$ denotes pullback by inclusion and
\[
\quad U=\left(\begin{array}{c} u\\\omega_{0,1}\end{array}\right),\quad U'=\left(\begin{array}{c} u'\\\omega_{0,1}'\end{array}\right).
\]
\end{lem}

Our goal will be to reduce to the case when $V$ is in fact diagonal. This is the situation first studied by Bukhgeim in the planar case in \cite{bu-2d}, laying the foundation to the manifold version in Proposition 2.5 in \cite{agtu2013}. In our case, due to the modified phase function $\Phi$ that we must resort to, our version of that proposition reads:

\begin{prop}\label{p:diagsols}
Let  $\Phi_{\hat p}$ be a holomorphic Morse function as in Lemma \ref{l:qPhip}.
\begin{enumerate}
\item
If
\[
V=\left[\begin{array}{cc}\tilde Q & 0 \\ 0 & \tilde F\end{array}\right],\quad \tilde Q\in L^\infty(M), \tilde F\in W^{1,p}(M),
\]
for some $p>2$, then there exist solutions to $(D+V)F_h=0$ on $M$ of the form
\[
F_h=\left(\begin{array}{c}
\e^{\Phi_{\hat p}}r_h \\ \e^{\bar\Phi_{\hat p}}(b+s_h),
\end{array}\right)
\]
for any anti-holomorphic one form $b$ and so that for some $\eps>0$
\begin{equation}\label{rem-ests}
\|r_h\|_{L^2(M)}+\|s_h\|_{L^2(M)}\leq\frac{Ch^{1/2+\eps}}{\delta^4}.
\end{equation}
\item
If instead $\tilde Q\in W^{1,p}(M)$ for some $p>2$ and $\tilde F\in L^\infty(M)$, then there exists solutions to $(D+V)G_h=0$ on $M$ of the form
\[
G_h=\left(\begin{array}{c}
\e^{\Phi_{\hat p}}(a+r_h) \\ \e^{\bar\Phi_{\hat p}}s_h,
\end{array}\right)
\]
for any holomorphic function $a$ and so that for some $\eps>0$, \eqref{rem-ests} still hold.
\end{enumerate}
\end{prop}
\begin{proof}[Sketch of proof] The proof closely resembles the one Proposition 2.5 in \cite{gt-gafa} but makes use of the refined Lemma \ref{l:dbpi} (or similar) instead of Lemma 2.2 in \cite{gt-gafa} (and it's variants respectively). There are some details that we in particular would like to highlight.

We make use of that the mentioned lemmas contain very explicit expressions for the remainder terms $r_h\in\Lambda^0(M), s_h\in\Lambda^{0,1}(M)$. It can be seen, as in \cite{gt-gafa}, that the remainder terms must solve the system
\begin{equation}\label{remsyst}
\begin{cases}
r_h+\dbar_\psi^{-1}(\tilde Fs_h)=-\dbar_\psi^{-1}(\tilde Fb)\\
s_h+\dbar_\psi^{*-1}(\tilde Qr_h)=-\dbar_\psi^{*-1}(\tilde Qa).
\end{cases}
\end{equation}
In the case $\tilde Q\in L^\infty(M), \tilde F\in W^{1,p}(M)$ we can choose $a=0$ and it follows that $r_h$ must satisfy
\begin{equation}\label{r_h-eq}
(I-S_h)r_h=-\dbar_\psi^{-1}(\tilde Fb),\quad S_h:=\dbar_\psi^{-1}\tilde F\dbar_\psi^{*-1}\tilde Q.
\end{equation}
The idea is now to solve through a Neumann series. From Lemma \ref{l:dbarpsiinv} and \ref{l:dbarpsi*inv} it follows that $\|S_h\|_{L^{2}\to L^{2}}=O(h^{1/2-\eps}), 0<\eps<1/2$, c.f. Lemma 2.4 in \cite{agtu2013}. We remark that also this result (or the corresponding Lemma 3.1 in \cite{gt-gafa}) must be modified slightly. However, as we will later require that $h^\eps\delta^{-7/2}$ to be small for some (preferebly as large as possible) $0<\eps<1/2$. This will ensure that the bound:
\[
\|S_h\|_{L^2\to L^2}\leq\frac{Ch^{1/2-\eps}}{\delta^4}
\]
still makes the Neumann series argument valid. Furthermore we may establish the estimate \eqref{rem-ests} in an analogous way. So we can solve equation \eqref{r_h-eq} for small $h$,
\begin{equation}\label{r_h}
r_h=-\sum_{j=0}^{\infty}S_h^j\dbar_\psi^{-1}\tilde Fb.
\end{equation}
The above $r_h\in L^{q}(M)$ for any $q\geq2$ and substituting this solution into the equation for $s_h$ in \eqref{remsyst} we get
\begin{equation}\label{s_h}
s_h=-\dbar_\psi^{*-1}(\tilde Qr_h).
\end{equation}
Now another application of Lemma \ref{l:dbarpsiinv} and \ref{l:dbarpsi*inv} gives the $L^{2}$-estimates in Proposition \ref{p:diagsols}. The case of $\tilde Q\in W^{1,p}(M), \tilde F\in L^\infty(M)$ is proved by a similar argument after choosing $b=0$.
\end{proof}



\subsection{The distance between Cauchy data for systems}
\label{s:dist-cd}
We would like to compare the Cauchy data for different potentials $V_1,V_2$ in a meaningful and quantitative manner. One standard method is to use a pseudo-distance inspired by the so-called Hausdorff distance.

Recall the definition of the Cauchy data spaces $\mathcal{C}_{L_j}$, as in \eqref{CauchyData}. Assuming that $u_j\in H^k(M)$ solves
\begin{equation}
\begin{cases}
L_ju_j=0,\\
u_j|_{\del M}=f_j,
\end{cases}
\end{equation} 
for $f_j\in H^{k-1/2}(\del M)$, for some $k\geq1$. Then $ \nabla_\nu^{X_j}u_j=g_j\in H^{k-3/2}(\del M)$ and we may consider a norm on $\mathcal{C}_{L_j}$ defined by,
\[
\|(f_j,g_j)\|_{H^{s}(\del M)\oplus H^{s-1}(\del M)}:=\|f_j\|_{H^{s}(\del M)}+\|g_j\|_{H^{s-1}(\del M)},\quad s\geq1/2.
\]
Then we set, for $(f_j,g_j)\in\mathcal{C}_{L_j}$,
\[
d((f_1,g_1),(f_2,g_2)):=\frac{\|(f_1,g_1)-(f_2,g_2)\|_{H^{s}(\del M)\oplus H^{s-1}(\del M)}}{\|f_1\|_{H^{s}(\del M)}}
\]
and define
\[
d(\mathcal{C}_{L_1},\mathcal{C}_{L_2}):=\max\left\{\sup_{\mathcal{C}_{L_1}}\inf_{\mathcal{C}_{L_2}} d((f_1,g_1),(f_2,g_2)),\sup_{\mathcal{C}_{L_2}}\inf_{\mathcal{C}_{L_1}} d((f_2,g_2),(f_1,g_1))\right\}.
\]

Correspondingly for the system formulation, we think of the Cauchy data $\mathcal{C}_V$ is made up of boundary values $\iota_{\del M}^*(u,\star\omega)^T$ for $H^{k}$-solutions $U=(u,\omega)^T$ to $(D+V)U=0$. We may consider $\mathcal{C}_V$ being a subset of $H^{s}(\del M)\oplus H^{s-1}(\del M)$, whose norm we can use to introduce a distance when considering two potentials. Suppose we have two traces, $(f_j,\lambda_j)^T\in\mathcal{C}_{V_j}, j=1,2$, then we can consider the quantity
\[
d_{\del M}((f_1,\lambda_1),(f_2,\lambda_2)):=\frac{\|(f_1,\lambda_1)-(f_2,\lambda_2)\|_{H^{s}(\Sigma(\del M))}}{\|f_1\|_{H^{s}(\del M)}}.
\]
The $H^{s}(\Sigma(\del M))$-norm is defined, for $s\geq1/2$, by
\[
\|(f,\lambda)\|_{H^{s}(\Sigma(\del M))}:=\|f\|_{H^{s}(\del M)}+\|\lambda\|_{H^{s-1}(\Lambda^1(\del M))},\quad U=(f,\lambda)^T\in\Sigma(\del M).
\]
Then we can define a distance between Cauchy data as
\begin{multline*}
d'(\mathcal{C}_{V_1},\mathcal{C}_{V_2})\\=\max\left\{\sup_{\mathcal{C}_{V_1}}\inf_{\mathcal{C}_{V_2}} d_{\del M}((f_1,\lambda_1),(f_2,\lambda_2)),\sup_{\mathcal{C}_{V_2}}\inf_{\mathcal{C}_{V_1}} d_{\del M}((f_2,\lambda_2),(f_1,\lambda_1))\right\}.
\end{multline*}
\begin{prop}\label{CDcmp1}
If $(D+V_j)U_j=0$ are the system formulations of the problems $L_ju_j=0$, then
\[
d'(\mathcal{C}_{V_1},\mathcal{C}_{V_2})\leq d(\mathcal{C}_{L_1},\mathcal{C}_{L_2}).
\]
\end{prop}
\begin{proof}
Pick any $(f_j,g_j)\in\mathcal{C}_{L_j}$, then $f_j=u_j|_{\del M}, g_j=\nabla_\nu^{X_j}u_j|_{\del M}$ where $u_j\in H^k(M)$ solves $L_ju_j=0$. Then the corresponding solutions $U_j=(u_j,\omega_j)$ to $(D+V_j)U_j=0$ has Cauchy data $(u_j|_{\del M},\iota_{\del M}^*\star\omega_{j})=(u_j|_{\del M}, \iota_{\del M}^* \pi_{0,1}\nabla^{X_j}u_j(\nu))$. So clearly,
\[
\|\iota_{\del M}^*\star\omega_j\|_{H^{k-3/2}(\Lambda^1(\del M))}\leq\|\nabla_\nu^{X_j}u_j|_{\del M}\|_{H^{k-3/2}(\del M)},
\]
so $d_{\del M}((f_1,\iota_{\del M}^*\star\omega_1),(f_2,\iota_{\del M}^*\star\omega_2))\leq d((f_1,g_1),(f_2,g_2))$.
\end{proof}

We will from now on use the notation $d(\mathcal{C}_1,\mathcal{C}_2)=d(\mathcal{C}_{L_1},\mathcal{C}_{L_2})$ and not bother much with estimates that could be expressed with $d'(\mathcal{C}_{V_1},\mathcal{C}_{V_2})$ instead. However, we will still need the definition of $d'$ when we later will solve a diagonalized version of the problems $(D+V_j)U_j=0$ and in that case Proposition \ref{CDcmp1} will become weaker.

Furthermore, we will from now on only consider the case when $k=1$ (or $s=1/2$), in which case
\begin{align*}
d((f_1,g_1),(f_2,g_2))&=\frac{\|f_1-f_2\|_{H^{1/2}(\del M)}+\|g_1-g_2\|_{H^{-1/2}(\del M)}}{\|f_1\|_{H^{1/2}(\del M)}},\\
d_{\del M}((f_1,\lambda_1),(f_2,\lambda_2))&=\frac{\|f_1-f_2\|_{H^{1/2}(\del M)}+\|\lambda_1-\lambda_2\|_{H^{-1/2}(\Lambda^1(\del M))}}{\|f_1\|_{H^{1/2}(\del M)}}.
\end{align*}

%% file: bdry.tex
\section{System reduction and estimates}\label{sec:bdry}
Suppose now that we are given two magnetic Schrödinger operators $L_j:=L_{X_j,q_j}, j=1,2.$ Introduce $A_j:=\pi_{0,1}X_j, B_j:=\pi_{1,0}X_j$, where $\pi_{0,1}$ and $\pi_{1,0}$ are the projections discussed in Section \ref{sec:riemann}. We first observe that we can rewrite $L_j$ from the form \eqref{SchrodingerX} into
\begin{equation}\label{rwSchrodinger}
L_j=-2\i\star(\del+\i \bar A_j\wedge)(\dbar+\i A_j)+Q_j,\quad Q_j=-\star\d X_j+q_j.
\end{equation}
If  $\alpha_j$ are primitive functions in the sense that $\dbar\alpha_j=A_j$ (the existence of such $\alpha_j$ is guaranteed by Lemma \ref{l:dbarinv}), then we can furthermore rewrite \eqref{rwSchrodinger} (using integrating factor) to
\begin{equation}\label{rwSchrodinger2}
L_j=2\e^{-\i\bar\alpha_j}\dbar^*\e^{\i\bar\alpha_j}\e^{-\i\alpha_j}\dbar\e^{\i\alpha_j}+Q_j.
\end{equation}
In order to abbreviate, let us denote by $F_j=\e^{\i\alpha_j}$, then we can once again rewrite \eqref{rwSchrodinger2} as
\begin{equation}\label{rwSchrodinger3}
L_j=2\bar F_j\dbar^*\bar F_j^{-1}F_j^{-1}\dbar F_j+Q_j,\quad Q_j=-\star\d X_j+q_j.
\end{equation}

Introducing $\omega=(\dbar+\i A_j)u$ we see that we can rewrite the second order partial differential equations $L_ju=0$ as the first order $\dbar$-system
\[
\left[\begin{array}{cc} Q_j/2 & -\i\star(\del+\i\bar A_j\wedge\cdot)\\\dbar+\i A_j & -1\end{array}\right]\left(\begin{array}{c} u\\\omega\end{array}\right)=\left(\begin{array}{c}0\\0\end{array}\right),
\]
or equivalently
\begin{equation}\label{dbarreduction1}
\left[\begin{array}{cc} 0 & \dbar^{*}\\\dbar&0\end{array}\right]\left(\begin{array}{c} u\\\omega\end{array}\right)+\left[\begin{array}{cc} Q_j/2 & \star(\bar A_j\wedge\cdot)\\ \i A_j&-1\end{array}\right]\left(\begin{array}{c} u\\\omega\end{array}\right)=\left(\begin{array}{c}0\\0\end{array}\right).
\end{equation}
Using notations from Section \ref{sec:systems} we can abbreviate \eqref{dbarreduction1} by $(D+V_j)U=0$. By a similar argument leading to the form \eqref{rwSchrodinger3} of $L_j$ we can also observe that we can split the system \eqref{dbarreduction1} further into (c.f. \cite{agtu2013})
\begin{multline*}
\left[\begin{array}{cc} \bar F_j & 0\\ 0 & F_j^{-1}\end{array}\right]\left(\left[\begin{array}{cc} 0 & \dbar^{*}\\\dbar&0\end{array}\right]+\left[\begin{array}{cc} \bar F_j^{-1}Q_jF_j^{-1}/2 & 0\\ 0 &-F_j\bar F_j\end{array}\right]\right)\\\times\left[\begin{array}{cc} F_j & 0\\ 0 & \bar F_j^{-1}\end{array}\right]\left(\begin{array}{c} u\\\omega\end{array}\right)=\left(\begin{array}{c}0\\0\end{array}\right)
\end{multline*}
or equivalently, since the leftmost matrix is invertible
\begin{equation}\label{dbarreduction2}
\left(\left[\begin{array}{cc} 0 & \dbar^{*}\\\dbar&0\end{array}\right]+\left[\begin{array}{cc} |F_j|^{-2}Q_j/2 & 0\\ 0 &-|F_j|^{2}\end{array}\right]\right)\left(\begin{array}{c} \tilde u\\\tilde\omega\end{array}\right)=\left(\begin{array}{c}0\\0\end{array}\right),
\end{equation}
where
\[
\tilde U=\left(\begin{array}{c} \tilde u\\\tilde\omega\end{array}\right)=\left(\begin{array}{c} F_ju\\\bar F_j^{-1}\omega\end{array}\right)\in\Sigma(M).
\]
Denoting the potential matrix in \eqref{dbarreduction2} by $\tilde V_j$, the system can be abbreviated as $(D+\tilde V_j)\tilde U=0$. Now we are in the case discussed above, with diagonal potential matrices, that was treated by Bukhgeim.

Suppose now that we have a solution to $(D+\tilde V_1)\tilde U_h^{1}=0$, then by Proposition \ref{p:diagsols} we can assume that the solution has the form
\[
\tilde U_h^{1}=\left(\begin{array}{c} \e^{\Phi/h}(a+r_h^{1}) \\ \e^{\bar\Phi/h}s_h^{1} \end{array}\right).
\]
It follows that a solution to $(D+V_1)U_h^{1}=0$ can be found on the form
\begin{equation}\label{systsol1}
U_h^{1}=\left(\begin{array}{c} F_1^{-1}\e^{\Phi/h}(a+r_h^{1}) \\ \bar F_1\e^{\bar\Phi/h}s_h^{1} \end{array}\right),
\end{equation}
where $a$ is an arbitrary holomorphic function.

Similarly, a solution to $(D+V_2^{*})U_h^{2}=(D+V_2)U_h^{2}=0$ can be found on the form
\begin{equation}\label{systsol2}
U_h^{2}=\left(\begin{array}{c} F_2^{-1}\e^{-\Phi/h}r_h^{2} \\ \bar F_2\e^{-\bar\Phi/h}(b+s_h^{2}) \end{array}\right).
\end{equation}

Suppose now that $U_h$ is a solution to $(D+V_2)U_h=0$, and $U_h^{1}, U_h^{2}$ are the solutions described above. Then by Lemma \ref{boundaryintid}
\begin{multline}\label{bdryint1}
\langle(D+V_2)U_h^{1},U_h^{2}\rangle_{L^2(\Sigma(M))}=\langle(D+V_2)(U_h^{1}-U_h),U_h^{2}\rangle_{L^2(\Sigma(M))}\\
=\langle U_h^{1}-U_h,U_h^{2}\rangle_{\del M}.
\end{multline}
At the same time it is also true that
\begin{equation*}
\langle(D+V_2)U_h^{1},U_h^{2}\rangle_{L^2(\Sigma(M))}=\langle(V_2-V_1)U_h^{1},U_h^{2}\rangle_{L^2(\Sigma(M))}.
\end{equation*}
So we have derived the boundary integral identity
\begin{equation}\label{bdryint2}
\langle(V_2-V_1)U_h^{1},U_h^{2}\rangle_{L^2(\Sigma(M))}=\langle U_h^{1}-U_h,U_h^{2}\rangle_{\del M},
\end{equation}
from which we will derive an estimate in order to compare Cauchy data for the potential matrices $V_j$ and their diagonalized counterparts $\tilde V_j$, in the sense of \eqref{dbarreduction1}-\eqref{dbarreduction2}. First we show the following auxiliary estimate.
\begin{lem}\label{l:auxest}
Assume that there is a constant $K>0$ such that for some $p>2$,
\begin{equation}\label{aprioriest}
\max\{\|q_j\|_{W^{1,p}(M)},\|X_j\|_{W^{2,p}(\Lambda^{1}(M))}\}\leq K. 
\end{equation}
and consider the systems corresponding to the problems $L_ju=0$ as described in \eqref{dbarreduction1}. Then the boundary integral identity \eqref{bdryint2} implies the following inequality for small $h>0,\delta>0$:
\begin{equation}\label{auxest1}
\left|\int_M F_1^{-1}F_2a\,(A_1-A_2)\wedge\star\overline{b}\right|\leq C\left(\frac{h^{1/2+\eps}}{\delta^4}+d(\mathcal{C}_{1},\mathcal{C}_{2})\|U_h^1\|_{H^1}\|U_h^2\|_{H^1}\right),
\end{equation}
where $C=C(K,M,p), c>0, A_j=\pi_{0,1}X_j, d(\mathcal{C}_{1},\mathcal{C}_{2})$ is the distance between Cauchy data for the problems $L_ju=0$ and $a,b,F_j$ are the quantities defined above appearing in the solutions $U_h^j, j=1,2$ of the systems.
\end{lem}

\begin{rem} Let us first remark on the a priori estimate \eqref{aprioriest}. Under the assumptions
\begin{align}
A_j&\in C^{1+r}\cap W^{s,p}(\Lambda^{0,1}(M)),\quad p>1,r+s>1,r\notin\mathbf{N}, sp>4,\label{aprioriA}\\
Q_j&\in W^{1,q}(M),\quad q>2,\label{aprioriQ}
\end{align}
we will have that $\tilde V_j\in W^{1,q}(\text{End}(\Sigma(M)))$, c.f.\ \cite{agtu2013}.

Furthermore, the assumption
\begin{equation*}
\max\{\|q_j\|_{W^{1,p}(M)},\|X_j\|_{W^{2,p}(\Lambda^{1}(M))}\}\leq K,\quad\textrm{for some}\quad p>2. 
\end{equation*}
implies in particular that
\begin{equation}\label{aprioriest2}
\max\{\|Q_j\|_{W^{1,p}(M)},\|A_j\|_{W^{1,p}(\Lambda^{0,1}(M))}\}\leq K
\end{equation}
holds. Then by Sobolev embedding and elliptic regularity (for $\dbar$, see e.g. Theorem 4.6.9 in \cite{Krantz1982}), it follows that
\begin{align}
\max\{\|Q_j\|_{L^\infty(M)},\|A_j\|_{L^\infty(\Lambda^{0,1}(M))}\}&\leq K,\label{potbds1}\\
\|\alpha_j\|_{L^\infty(M)}\leq C_0\|A_j\|_{W^{1,p}(\Lambda^{0,1}(M))}&\leq C_0K\leq C,\label{potbds2}\\
\|F_j\|_{L^\infty(M)}\leq\e^{C_0K}&\leq C,\quad C=C(K,M,p).\label{potbds3}
\end{align}
\end{rem}

\begin{proof}[Proof of Lemma \ref{l:auxest}]
Expanding the left hand side in \eqref{bdryint2}, we find
\begin{multline}
\langle(V_2-V_1)U_h^{1},U_h^{2}\rangle_{L^2(\Sigma(M))}\\
=\frac12\int_M \e^{2\i\psi/h}F_1^{-1}\bar F_2^{-1}Q(a+r_h^{1})\overline{r_h^{2}}+2\bar F_1\bar F_2^{-1}\star(\bar A\wedge s_h^{1})\overline{r_h^{2}} \,\d V_g\\
+\i\int_M F_1^{-1}F_2(a+r_h^{1})A\wedge\star\overline{(b+s_h^{2})}.
\end{multline}
where $A=A_2-A_1, Q=Q_2-Q_1, a$ can be any holomorphic function and $b$ any antiholomorphic 1-form. The next step will be to use the estimates from Proposition \ref{p:diagsols}. Under the assumption in \eqref{aprioriest} and the following remark, we get (by applying Cauchy-Schwarz inequality) the estimates
\begin{align}
\left|\int_M \e^{2\i\psi/h}F_1^{-1}\bar F_2Q(a+r_h^{1})\overline{r_h^{2}}\,\d V_g\right|&\leq C\left(\frac{h^{1/2+\eps}}{\delta^4}+\frac{h^{1+2\eps}}{\delta^8}\right)\leq \frac{Ch^{1/2+\eps}}{\delta^4},\label{potintest1}\\
\left|\int_M \bar F_1\bar F_2^{-1}\star(\bar A\wedge s_h^{1})\overline{r_h^{2}}\,\d V_g\right|&\leq \frac{Ch^{1+2\eps}}{\delta^8}\leq \frac{Ch^{1/2+\eps}}{\delta^4},\label{potintest2}\\
\left|\int_M F_1^{-1}F_2(a+r_h^{1})A\wedge\star\overline{(b+s_h^{2})}\right|&\leq\left|\int_M aF_1^{-1}F_2A\wedge\star\overline{b}\right|+\frac{Ch^{1/2+\eps}}{\delta^4}.\label{potintest3}
\end{align}

By examining the boundary term in \eqref{bdryint2} we find
\[
\langle U_h^{1}-U_h,U_h^{2}\rangle_{\del M}=\int_{\del M}\iota_{\del M}^{*}\left((u_1-u)\star\overline{\omega_2}-\star(\omega_1-\omega)\overline{u_2}\right),
\]
where we temporarily have abbreviated the solutions according to the earlier convention, $U_h^{j}=(u_j,\omega_j)^{T}$. By Cauchy-Schwarz inequality,
\begin{multline*}
|\langle U_h^{1}-U_h,U_h^{2}\rangle_{\del M}|\\
\leq2\|(u_1-u,\nabla_\nu^{X_1}u_1-\nabla_\nu^{X_2}u)|_{\del M}\|_{H^{1/2}(\del M)\oplus H^{-1/2}(\del M)}\|\iota_{\del M}^*U_h^2\|_{H^{1/2}(\del M)}\\
\leq Cd((u_1,\nabla_\nu^{X_1}u_1),(u,\nabla_\nu^{X_2}u))\|U_h^1\|_{H^{1}(\Sigma(M))}\|U_h^2\|_{H^{1}(\Sigma(M))}
\end{multline*}
by the boundedness of the trace operator. Since this holds for any solution $U_h$ to $(D+V_2)U_h=0$ it follows by taking infimum over the corresponding Cauchy data space that
\begin{multline}\label{bdryest1}
|\langle(V_2-V_1)U_h^1,U_h^2\rangle_{L^2(\Sigma(M))}|
\leq Cd(\mathcal{C}_{1},\mathcal{C}_{2})\|U_h^1\|_{H^{1}(\Sigma(M))}\|U_h^2\|_{H^{1}(\Sigma(M))}.
\end{multline}
The proof is finished by rearranging the terms in \eqref{bdryint2} and applying the triangle inequality.
\end{proof}

The next step is to estimate the Sobolev norms of the solutions appearing in the right hand side of \eqref{bdryest1}. This requires a quite detailed discussion but will yield results that we will use more than once.

\subsection{$H^{1}$-estimates of solutions to $(D+V)U=0$}\label{sec:solest}
We will need to estimate solutions of systems $(D+V)U=0$ in $H^{1}$-norm sense. For the general (non-diagonal) potentials $V$, that we must consider, we saw that solutions were given by \eqref{systsol1} and \eqref{systsol2}. Our goal will be to establish:
\begin{prop}
Let $U_1,U_2$ be the solutions given in Proposition \ref{p:diagsols} with $\tilde Q,\tilde F\in W^{1,p}(M), p>2$. Then there are constants $c>0, D=D(K,M,p)>0$ such that
\begin{equation}\label{solutionests2}
\max\{\|U_1\|_{H^{1}(M,\Sigma(M))},\|U_2\|_{H^{1}(M,\Sigma(M))}\}\leq D\e^{c/h},
\end{equation}
for small $h>0$.
\end{prop}

\begin{proof}
Let us first recall what we mean with the $H^{1}$-norm of some $U=(u,\omega_{0,1})^T\in\Sigma(M)$.
\begin{align*}
\|U\|_{H^{1}(\Sigma(M))}^2:&=\|u\|_{H^{1}(M)}^2+\|\omega_{0,1}\|_{H^{1}(\Lambda^{1}(M))}^2,\\
\|u\|_{H^{1}(M)}^2&=\|u\|_{L^{2}(M)}^2+\|\nabla u\|_{L^{2}(M)}^2,\\
\|\omega_{0,1}\|_{H^{1}(\Lambda^{1}(M))}^2&=\|\omega_{0,1}\|_{L^{2}(\Lambda^{1}(M))}^2+\|\nabla\omega_{0,1}\|_{L^{2}(\Lambda^{1}(M))}^2.
\end{align*}
For the solutions in \eqref{systsol1} and \eqref{systsol2}, this means that we need to estimate
\begin{align*}
&\|F_1^{-1}\e^{\Phi/h}(a+r_h^{1})\|_{H^{1}(M)},\quad\|\bar F_1\e^{\bar\Phi/h}s_h^{1}\|_{H^{1}(\Lambda^{1}(M))},\\
&\|F_2^{-1}\e^{-\Phi/h}r_h^{2}\|_{H^{1}(M)},\quad\|\bar F_2\e^{-\bar\Phi/h}(b+s_h^{2})\|_{H^{1}(\Lambda^{1}(M))},
\end{align*}
where we are free to choose the holomorphic function $a$ and antiholomorphic 1-form $b$. We start by observing that there is a constant $c>0$ so that
\[
|\e^{\pm\Phi/h}|\leq\e^{c/h},
\]
since $\varphi=\text{Re}\,\Phi$ is harmonic. By the a priori assumptions on the $A_j$:s in \eqref{aprioriA} and the discussion leading to \eqref{potbds2}, we will also have no trouble bounding the first order partial derivatives of the $F_j$ by some constant depending on the a priori bounding constant $K$. 

We need to be a bit careful with the remainders $r_h^{j}, s_h^{j}$. By studying the system \eqref{remsyst} and invoking elliptic regularity we can however see that we are in no danger, assuming sufficient regularity on $\tilde Q,\tilde F$. Rewriting the system we see that,
\[
\begin{cases}
r_h=-\dbar_\psi^{-1}(\tilde F(b+s_h)),\\
s_h=-\dbar_\psi^{*-1}(\tilde Q(a+r_h)).
\end{cases}
\]
Now it is more clear that the right hand side should belong to $H^1(M)$. To get an estimate, observe that by writing out the operators as in Section 2, $\dbar_\psi^{-1}=R\dbar^{-1}\e^{-2\i\psi/h}E, \dbar_\psi^{*-1}=R'\dbar^{*-1}\e^{2\i\psi/h}E'$ we have (on $M$) that for small enough $h>0$ (by Lemma \ref{l:dbarinv}(i), \ref{l:dbar*inv}(i) and Proposition \ref{p:diagsols}),
\begin{align*}
\|\dbar r_h\|_{L^2}&=\|\dbar\dbar^{-1}\e^{-2\i\psi/h}E(\tilde F(b+s_h))\|_{L^2}\leq CK(\|b\|_{L^2}+\|r_h\|_{L^2})\leq C,\\
\|\del s_h\|_{L^2}&=\|\del\dbar^{*-1}\e^{2\i\psi/h}E'(\tilde Q(a+r_h))\|_{L^2}\leq CK(\|b\|_{L^2}+\|r_h\|_{L^2})\leq C.
\end{align*}
Furthermore,
\begin{align*}
\|\del r_h\|_{L^2}&=\|\del\dbar^{-1}\e^{-2\i\psi/h}E(\tilde F(b+s_h))\|_{L^2}\leq CK(\|b\|_{L^2}+\|r_h\|_{L^2})\leq C,\\
\|\dbar s_h\|_{L^2}&=\|\dbar\del^{-1}\star\e^{2\i\psi/h}E(\tilde Q(a+r_h))\|_{L^2}\leq CK(\|b\|_{L^2}+\|r_h\|_{L^2})\leq C.
\end{align*}
since $\del\dbar^{-1},\dbar\del^{-1}$ are bounded operators (related to the so-called Beurling transform) on $L^p(M)$. In local coordinates the kernels are of Calder\'on-Zygmund type, modulo smoothing terms by Lemma \ref{l:dbarinv} and \ref{l:dbar*inv}, c.f. \cite{Vekua1962}.
\end{proof}


\begin{rem}
Continuing the argument in the above proof it is in fact also more or less almost immediate that if $\tilde Q,\tilde F\in W^{k,p}(M), k\geq1, p>2$, then
\[
\|r_h\|_{H^k(M)}+\|s_h\|_{H^k(M)}=O(h^{1-k}).
\]
\end{rem}

\subsection{Conclusion of the system reduction step}

Adding up \eqref{bdryint1}-\eqref{solutionests2} we have managed to show that
\begin{equation}\label{auxest1}
\left|\int_M aF_1F_2^{-1}A\wedge\star\overline{b}\right|\leq C\frac{h^{1/2+\eps}}{\delta^4}+D\e^{c/h}d(\mathcal{C}_{1},\mathcal{C}_{2}).
\end{equation}
Choosing for some $1-\eps<\alpha<1$
\[
h=\frac{c}{\alpha|\log d(\mathcal{C}_{1},\mathcal{C}_{2})|},
\]
we get the estimate (for some large enough $C=C(K,M,p)>0$)
\begin{equation}\label{auxest2}
\left|\int_M aF_1^{-1}F_2A\wedge\star\overline{b}\right|\leq \frac{C}{\delta^4|\log d(\mathcal{C}_{1},\mathcal{C}_{2})|^{1/2+\eps}}.
\end{equation}
Looking at the left hand side-integral we also observe that
\begin{multline*}
-\i\int_M F_2(A_2-A_1)F_1^{-1}a\wedge\star\overline{b}
=\langle\dbar(F_2F_1^{-1}a),b\rangle_{L^2(\Lambda^{1}( M))}\\
=\int_{\del M}\iota_{\del M}^*\left(F_2F_1^{-1}a\star\overline{b}\right)+\langle F_2F_1^{-1}a,\dbar^*b\rangle_{L^2(M)}=\int_{\del M}\iota_{\del M}^*\left(F_2F_1^{-1}a\star\overline{b}\right).
\end{multline*}
The first equality follows from the fact that $F_j^{-1}\dbar F_j=-F_j\dbar F_j^{-1}=A_j$ by construction, while the third equality is just Green's integral identity. Finally, the last equality follows since $b$ is antiholomorphic. Thus we have shown
\begin{lem}\label{l:systred1}
If for some $p>2$, an a priori assumption of type
\[
\max\{\|q_j\|_{W^{1,p}(M)},\|X_j\|_{W^{2,p}(M,\Lambda^{0,1}(\bar M))}\}\leq K
\]
hold, then for some small $\eps>0$,
\begin{equation}\label{auxest3}
\left|\int_{\del M}\iota_{\del M}^*\left(F_2F_1^{-1}a\star\overline{b}\right)\right|\leq \frac{C(K,M,p)}{\delta^4|\log d(\mathcal{C}_{1},\mathcal{C}_{2})|^{1/2+\eps}},
\end{equation}
where $\mathcal{C}_{j}$ is the Cauchy data associated with the operators $L_j=L_{X_j,q_j}$, as defined in \eqref{cauchysystem}, and $a\in\mathcal{H}(M), b\in\Lambda^{0,1}(M)$.
\end{lem}
Later we will choose $\delta = (\log|\log d(\mathcal{C}_{1},\mathcal{C}_{2})|)^{-\eps'}$ for some $\eps'<1/2+\eps$ so that \eqref{auxest3} can be replaced by
\begin{equation}\label{auxest31}
\left|\int_{\del M}\iota_{\del M}^*\left(F_2F_1^{-1}a\star\overline{b}\right)\right|\leq \frac{C}{|\log d(\mathcal{C}_{1},\mathcal{C}_{2})|^{\beta}},
\end{equation}
with $0<\beta<1/2+\eps-\eps'\gamma'$ where $\gamma'>0$ can be choosen arbitrarily small. We will from now on assume that this choice of $\delta$ has been made and use \eqref{auxest31} in place of \eqref{auxest3} when referring to Lemma \ref{l:systred1}.


%% file: system2.tex
\section{Reduction to diagonal system}\label{system2}

One of the key observations in \cite{agtu2013} is that if $\mathcal{C}_{1}=\mathcal{C}_{2}$, then choosing $a=1$
\[
\int_{\del M}\iota_{\del M}^*\left(F_2F_1^{-1}\star\overline{b}\right)=0.
\]
(This is also a consequence of Lemma \ref{l:systred1} of course.) The following lemma is then used extensively:
\begin{lem}[Lemma 2.8 in \cite{agtu2013}]\label{l:restholo}
A complex-valued function $f\in H^{1/2}(\del M)$ is the restriction of a holomorphic function if and only if for all 1-forms $\eta\in C^{\infty}(M,\Lambda^{1,0}(M))$ satisfying $\dbar\eta=0$,
\[
\int_{\del M}f\iota_{\del M}^*\eta=0.
\]
\end{lem}
The proof can be found in \cite{gt-gafa} where the above result is contained in Lemma 4.1. One must show that the harmonic extension of $f$ is actually holomorphic, and this can be done by considering the so-called Hodge-Morrey decomposition of $(0,1)$-forms.

Using the above lemma, the authors are thus able to conclude that $(F_2^{-1}F_1)|_{\del M}$ is indeed the restriction of a holomorphic function. Hence they may reduce the case of a general matrix potential $V_j$ to the diagonalized case $\tilde V_j$. 

Clearly, we are not in the same situation here so we need to motivate a similar reduction step in a slightly different way.

\subsection{Auxiliary estimate}

Let us introduce the subspace
\[
U=\{\omega=\iota_{\del M}^*\bar b;\dbar b=0, b\in\Lambda^1(M)\},
\]
of $X=H^{1/2}(\Lambda^1(\del M))$ (or $H^s(\Lambda^1(\del M))$, for any $s\geq0$) and argue abstractly from the viewpoint of Hilbert space theory. Consider the following subspace of the dual space $X^*$,
\[
\ker U:=\{x^*\in X^*;x^*(x)=0 \text{ for all } x\in U\}.
\]
Take a Schauder basis $\{x_n^*\}_{n=1}^\infty$ such that $\ker U=\text{span}\,\{x_n^*\}_{n=1}^\infty$ and a dual basis $\{x_n\}_{n=1}^\infty$ (defined by $x_n^*(x_m)=\delta_{mn}$). Now we claim that 
\[
X=U\oplus\text{span}\,\{x_n\}_{n=1}^\infty.
\]
Indeed, suppose there is $0\neq x^*\in X^*$ is such that $x^*(x)=0$ for all $x\in U\oplus\text{span}\,\{x_n\}_{n=1}^\infty$. Then in particular $x^*(x_n)=0$ for all $n\in\mathbf{N}$. Thus $x^*\notin\ker U$, but at the same time we must have $x^*(x)=0$ for all $x\in U$ and thus $x^*\in\ker U$. An obvious contradiction unless $x^*=0$. To see that it is indeed a direct sum, pick $x\in U\cap\text{span}\,\{x_n\}_{n=1}^\infty$, then $x=\sum a_n x_n\in U$, but then $0=x_n^*(x)=a_n$ for all $n\in\mathbf{N}$.

Now take another Schauder basis $\{y_n\}_{n=1}^\infty$ such that $\text{span}\,\{y_n\}_{n=1}^\infty=U$ and a dual basis $\{y_n^*\}_{n=1}^\infty\subset X^*$, then we can in a very similar way show that
\[
X^*=\ker U\oplus\text{span}\,\{y_n^*\}_{n=1}^\infty.
\]
Introducing the projection
\[
\pi:X^*\to\ker U
\]
we have by the above splitting of $X^*=H^{-1/2}(\Lambda^1(\del M))$ that any linear functional on $X$ may be written
\[
x^*=\pi x^*+(1-\pi)x^*.
\]
Now we claim that
\[
\|(1-\pi)x^*\|=\sup_{\|x\|\leq1}|(1-\pi)x^*(x)|=\sup_{\stackrel{\|x\|\leq1}{x\in U}}|(1-\pi)x^*(x)|.
\]
This equality follows since $(1-\pi)x^*(x_n)=0, n\in\mathbf{N}$ since $(1-\pi)x^*\in\text{span}\,\{y_n^*\}$.
In particular, if we consider the linear functional
\[
A_f:H^{1/2}(\Lambda^{1}(\del M))\to\C,\quad A_f[\omega]:=\int_{\del M}f\omega,\quad f\in H^{-1/2}(\del M),
\]
we can interpret Lemma \ref{l:systred1} as
\[
\|(1-\pi)(F_2F_1^{-1})|_{\del M}\|_{H^{-1/2}(\del M)}\leq\frac{C}{|\log d(\mathcal{C}_{1},\mathcal{C}_{2})|^{\beta}},
\]
since the operator norm of $A_{(1-\pi)(F_2F_1^{-1})|_{\del M}}$ equals the norm in the left hand side by the argument above. 
Furthermore, by the a priori assumptions in \eqref{aprioriA}-\eqref{potbds3}, we have for some $\delta'>0$, and interpolation
\[
\|(1-\pi)(F_2F_1^{-1})|_{\del M}\|_{H^{r}(\del M)}\leq\frac{C}{|\log d(\mathcal{C}_{1},\mathcal{C}_{2})|^{\gamma(\alpha,\beta)}},
\]
where $0<\alpha<1, 0<\gamma(\alpha,\beta)<(1-\alpha)\beta$ and $r>3/2$.

Now we have arrived at a stage where we have
\[
F_2F_1^{-1}=(1-\pi)F_2F_1^{-1}+\pi F_2F_1^{-1}
\]
and by Lemma \ref{l:restholo} we know that
\[
\pi F_2F_1^{-1}|_{\del M}=G|_{\del M},\quad\dbar G=0.
\]
Thus
\[
F_2|_{\del M}=F_1(1-\pi)F_2F_1^{-1}|_{\del M}+F_1G|_{\del M}
\]
which is equivalent with
\[
F_2|_{\del M}-\tilde F_1 |_{\del M}=F_1(1-\pi)F_2F_1^{-1}|_{\del M},\quad \tilde F_1=F_1G.
\]
Clearly it also holds that
\[
\dbar\tilde F_1=\i A_1\tilde F_1,\quad\dbar F_2=\i A_2{F_2}
\]
and by our calculations above and a priori assumptions on the $A_j$:s
\[
\|(F_2-\tilde F_1)|_{\del M}\|_{H^{r}(\del M)}\leq\frac{C}{|\log d(\mathcal{C}_{1},\mathcal{C}_{2})|^{\gamma}}.
\]
So by Sobolev embedding we can conclude the following lemma:
\begin{lem}\label{l:cdlog}
Let $\tilde F_1, F_2$ be as defined above, then there is a $0<\gamma\leq1/2$ such that
\begin{equation}\label{F-ests1}
\|(F_2-\tilde F_1)|_{\del M}\|_{C^{1}(\del M)}\leq\frac{C}{|\log d(\mathcal{C}_{1},\mathcal{C}_{2})|^{\gamma}},
\end{equation}
where $C=C(K,M,p)$ and $d(\mathcal{C}_{1},\mathcal{C}_{2})$ measures the distance between the Cauchy data spaces as defined above.
\end{lem}

\subsection{An inequalty between Cauchy data}
From \eqref{systsol1}-\eqref{systsol2} we see that solutions to systems $(D+V_j)U_j=0$ is related to the corresponding system $(D+\tilde V_j)\tilde U_j=0$ with diagonalized potentials $\tilde V_j$ (in the sense described above) via
\[
\tilde U_1=\left[\begin{array}{cc}\tilde{F}_1 & 0\\0 & \bar{\tilde{F}}_1^{-1}\end{array}\right] U_1,\quad\tilde U_2=\left[\begin{array}{cc}F_2 & 0\\0 & \bar{F}_2^{-1}\end{array}\right]U_2.
\]
The next lemma relates the distances for Cauchy data for the diagonalized potential $\tilde V_j$ and the Cauchy data for the corresponding partial differential equations with the pairs $(X_j,q_j)$.
\begin{lem}\label{l:syst-cauchy}
Suppose that we have reduced the boundary value problems for $L_ju_j=0$ to boundary value problems with diagonal potential matrices $(D+\tilde V_j)\tilde U_j=0$ as above. Then the distance of Cauchy data corresponding to the diagonalized problems is bounded in terms of the distance between Cauchy data for the initial problem according to 
\[
d'(\mathcal{C}_{\tilde V_1},\mathcal{C}_{\tilde V_2})\leq C\left(d(\mathcal{C}_{1},\mathcal{C}_{2})+\frac1{|\log d(\mathcal{C}_{1},\mathcal{C}_{2})|^{\gamma}}\right).
\]
The quantities $d'(\mathcal{C}_{\tilde V_1},\mathcal{C}_{\tilde V_2})$ and $d(\mathcal{C}_{1},\mathcal{C}_{2})$ are those defined in Section \ref{s:dist-cd} and $0<\gamma\leq1/2$.
\end{lem}
\begin{proof}
\begin{multline*}
\|(\tilde u_1|_{\del M},\iota_{\del M}^*\star\tilde\omega_1)-(\tilde u_2|_{\del M},\iota_{\del M}^*\star\tilde\omega_2)\|_{H^{1/2}(\Sigma(\del M))}\\
=\|(\tilde F_1 u_1-F_2u_2)|_{\del M}\|_{H^{1/2}(\del M)}+\|\iota_{\del M}^{*}\star(\bar{\tilde{F_1}}^{-1} \omega_1-\bar F_2^{-1}\omega_2)\|_{H^{-1/2}(\Lambda^1(\del M))}
\end{multline*}\label{partialest1}
Starting with the first term
\begin{multline}
\|(\tilde F_1 u_1-F_2u_2)|_{\del M}\|_{H^{1/2}(\del M)}
\leq\|\tilde F_1|_{\del M}\|_{C^{1}(\del M)}\|(u_1-u_2)|_{\del M}\|_{H^{1/2}(\del M)}\\+\|(\tilde F_1-F_2)|_{\del M}\|_{C^{1}(\del M)}\|u_2|_{\del M}\|_{H^{1/2}(\del M)}\\
\leq C\left(\|(u_1-u_2)|_{\del M}\|_{H^{1/2}(\del M)}+\frac{\|u_2|_{\del M}\|_{H^{1/2}(\del M)}}{|\log d(\mathcal{C}_{1},\mathcal{C}_{2})|^{\gamma}}\right),
\end{multline}
where the last inequality follows from Lemma \ref{l:cdlog}. Similarly, the second term satisfies
\begin{multline}\label{partialest2}
\|\iota_{\del M}^{*}\star(\bar{\tilde{F_1}}^{-1} \omega_1-\bar F_2^{-1}\omega_2)\|_{H^{1/2}(\Lambda^1(\del M))}\\
=\left\|\iota_{\del M}^{*}\left(\frac{\bar F_2\bar{\tilde{F_1}}\bar{\tilde{F_1}}^{-1}\star\omega_1-\bar F_2\bar{\tilde{F_1}}\bar F_2^{-1}\star\omega_2}{\bar F_2\bar{\tilde{F_1}}}\right)\right\|_{H^{1/2}(\Lambda^1(\del M))}\\
\leq C\left(\|\iota_{\del M}^{*}(\bar F_2\star(\omega_1-\omega_2))\|_{H^{1/2}(\Lambda^1(\del M))}+\|\iota_{\del M}^{*}((\bar F_2-\bar{\tilde{F}}_1)\star\omega_2)\|_{H^{1/2}(\Lambda^1(\del M))}\right)\\
\leq C\left(\|\iota_{\del M}^{*}\star(\omega_1-\omega_2)\|_{H^{1/2}(\Lambda^1(\del M))}+\frac{\|\iota_{\del M}^{*}\star\omega_2\|_{H^{1/2}(\Lambda^1(\del M))}}{|\log d(\mathcal{C}_{1},\mathcal{C}_{2})|^{\gamma}}\right).
\end{multline}
Adding \eqref{partialest1} and \eqref{partialest2} we get
\begin{multline*}
\|((\tilde u_1-\tilde u_2)|_{\del M},\iota_{\del M}^*\star(\tilde\omega_1-\tilde\omega_2))\|_{H^{1/2}(\Sigma(\del M))}\\
\leq C\|((u_1-u_2)|_{\del M},\iota_{\del M}^*\star(\omega_1-\omega_2))\|_{H^{1/2}(\Sigma(\del M))}\\+C\frac{\|(u_2|_{\del M},\iota_{\del M}^*\star\omega_2)\|_{H^{1/2}(\Sigma(\del M))}}{|\log d(\mathcal{C}_{1},\mathcal{C}_{2})|^{\gamma}}.
\end{multline*}
This implies that
\begin{multline*}
\frac{\|((\tilde u_1-\tilde u_2)|_{\del M},\iota_{\del M}^*\star(\tilde\omega_1-\tilde\omega_2))\|_{H^{1/2}(\Sigma(\del M))}}{\|\tilde u_2|_{\del M}\|_{H^{1/2}(\del M)}}\\
\leq C\left(\frac{\|((u_1-u_2)|_{\del M},\iota_{\del M}^*\star(\omega_1-\omega_2))\|_{H^{1/2}(\Sigma(\del M))}}{\|u_2|_{\del M}\|_{H^{1/2}(\del M)}}+\frac1{|\log d(\mathcal{C}_{1},\mathcal{C}_{2})|^{\gamma}}\right),
\end{multline*}
since $F_2^{-1}$ is bounded from below and 
$\|\iota_{\del M}^*\star\omega_2\|_{H^{-1/2}(\del M)}\leq C\|u_2|_{\del M}\|_{H^{1/2}(\del M)}$. As we have a bijective correspondence between solutions to the $L_j,V_j$ and $\tilde V_j$ problems, it follows that
\[
d'(\mathcal{C}_{\tilde V_1},\mathcal{C}_{\tilde V_2})\leq C\left(d(\mathcal{C}_{1},\mathcal{C}_{2})+\frac1{|\log d(\mathcal{C}_{1},\mathcal{C}_{2})|^{\gamma}}\right)
\]
which is what we wanted to prove.
\end{proof}

The important conclusion of Lemma \ref{l:syst-cauchy} is that small differences in Cauchy data for the non-diagonalized $V_j$-case implies small differences in Cauchy data for the diagonal counterparts. This will allow us to deduce stability by considering the latter case.

\subsection{Estimates for the potential $q$ and magnetic field $\d X$}

We are now ready to complete the proof of Theorem \ref{mainthm1} (and \ref{mainthm2}). Recall our a priori assumptions \eqref{main-apriori1}:
\[
\|q_j\|_{W^{1,p}(M)}\leq K, \quad\|X_j\|_{W^{2,p}(T^*M)}\leq K,\quad j=1,2.
\]
If $\mathcal{C}_j$ are the Cauchy data spaces as defined in \eqref{CauchyData} for the corresponding magnetic Schrödinger operators $L_j:=L_{X_j,q_j}$, as defined in \eqref{laplaceX}-\eqref{SchrodingerX}. Then if the distance $d(\mathcal{C}_1,\mathcal{C}_2)$ is small enough, 
there is an $\alpha\in(0,1/2)$ such that
\begin{equation}\label{mainthmest11}
\|q_1-q_2\|_{L^2(M)}+\|\d(X_1-X_2)\|_{L^2(\Lambda^2(M))}\leq\frac{C}{\log^\alpha\log\frac{1}{d(\mathcal{C}_1,\mathcal{C}_2)}},
\end{equation}
where $C=C(K,M,\alpha)$.

We will split up the proof into three parts, using our reduction to a diagonal system as described in the above sections. 

Suppose that $\tilde U_h^{1}, \tilde U_h^{2}$ are the earlier constructed $H^{1}$-solutions to
\[
(D+\tilde V_1)\tilde U_h^{1}=0,\quad (D+\tilde V_2^{*})\tilde U_h^{2}=0
\]
respectively and that $\tilde U_h$ solves $(D+\tilde V_2)\tilde U_h=0$. By Lemma \ref{boundaryintid} we then have
\[
\langle(\tilde V_2-\tilde V_1)\tilde U_h^{1},\tilde U_h^{2}\rangle_{L^2(\Sigma(\bar M))}=\langle \tilde U_h^{1}-\tilde U_h,\tilde U_h^{2}\rangle_{\del M}.
\]
To abbreviate we introduce
\[
\tilde{V}=\left[\begin{array}{cc}\tilde{Q} & 0\\ 0 &\tilde{F}\end{array}\right]=\left[\begin{array}{cc} |F_2|^{-2}Q_2/2-|F_1|^{-2}Q_1/2 & 0\\ 0 &-|F_2|^{2}+|F_1|^{2}\end{array}\right]=\tilde V_2-\tilde V_1.
\]

\begin{lem}\label{l:mainpf1}
For 
$p_0\in M\setminus N_{\sqrt\delta}$ there is an $\eps>0$ (that can be chosen uniformly with respect to $p_0$) such that
\[
\max\{|\tilde Q(p_0)|,|\tilde F(p_0)|\}\leq\frac{C}{\delta^7(\log(1/H))^{\eps}}, \text{ where } H=|\log d(\mathcal{C}_1,\mathcal{C}_2)|^{-\gamma},
\]
and $C=C(K,M,p,\gamma,\eps), \delta>0, 0<\gamma\leq1/2$.
\end{lem}

\begin{proof}
We will consider two cases, first when the solutions $\tilde U_h^{j}$ are of the forms
\[
\tilde U_h^{1}=\left(\begin{array}{c} \e^{\Phi/h}(a_1+r_h^{1}) \\ \e^{\bar\Phi/h}s_h^{1} \end{array}\right),\quad\tilde U_h^{2}=\left(\begin{array}{c} \e^{-\Phi/h}(a_2+r_h^{2}) \\ \e^{-\bar\Phi/h}s_h^{2} \end{array}\right),
\]
where $a_1, a_2$ are holomorphic functions. Then we have
\begin{multline}\label{stabest:1a}
 \langle\tilde V\tilde U_h^{1},\tilde U_h^{2}\rangle_{L^2(\Sigma(\bar M))}\\
 =\int_M \tilde{Q}\e^{2\i\psi/h}(a_1\bar a_2+a_1\overline{r_h^2}+r_h^1\bar a_2+r_h^1\overline{r_h^2})\,\d V_g+\tilde{F}\e^{-2\i\psi/h}s_h^1\wedge\star\overline{s_h^2}.
\end{multline}
In particular if $a_1=a_2=a$ we claim that we can estimate, for $h$ small,
\begin{equation}\label{remest:1a}
\int_M \tilde{Q}\e^{2\i\psi/h}(a\overline{r_h^2}+r_h^1\bar a+r_h^1\overline{r_h^2})\,\d V_g+\tilde{F}\e^{-2\i\psi/h}s_h^1\wedge\star\overline{s_h^2}\leq\frac{ C h^{1+\eps}}{\delta^8}.
\end{equation}
Furthermore, if $\psi$ has a non-degenerate stationary point at $p_0\in M$, we claim that
\begin{equation}\label{spest:1a}
\left|\int_M \tilde{Q}\e^{2\i\psi/h}|a|^2\,\d V_g-\frac{ch}{\delta}\e^{2\i\psi(p_0)/h}\tilde Q(p_0)|a(p_0)|^2\right|\leq \frac{Ch^2}{\delta^7},
\end{equation}
where $C$ depends on the a priori bounds on the potentials.

Similarly, if we instead consider solutions
\[
\tilde U_h^{1}=\left(\begin{array}{c} \e^{\Phi/h}r_h^{1} \\ \e^{\bar\Phi/h}(b_1+s_h^{1}) \end{array}\right),\quad\tilde U_h^{2}=\left(\begin{array}{c} \e^{-\Phi/h}r_h^{2} \\ \e^{-\bar\Phi/h}(b_2+s_h^{2}) \end{array}\right),
\]
where $b_1, b_2$ are antiholomorphic 1-forms. Then
\begin{multline}\label{stabest:1b}
 \langle\tilde V\tilde U_h^{1},\tilde U_h^{2}\rangle_{L^2(\Sigma(\bar M))}\\
 =\iint_M \tilde{Q}\e^{2\i\psi/h}r_h^1\overline{r_h^2}\,\d V_g+\tilde{F}\e^{-2\i\psi/h}(b_1\wedge\star\overline{b_2}+b_1\wedge\star\overline{s_h^2}+s_h^1\wedge\star\overline{b_2}+s_h^1\wedge\star\overline{s_h^2}).
\end{multline}
In particular if $b_1=b_2=b$ we claim that we can, similarly as for \eqref{remest:1a}, estimate, for $h$ small,
\begin{equation}\label{remest:1b}
\int_M \tilde{Q}\e^{2\i\psi/h}r_h^1\overline{r_h^2}\,\d V_g+\tilde{F}\e^{-2\i\psi/h}(b\wedge\star\overline{s_h^2}+s_h^1\wedge\star\overline{b}+s_h^1\wedge\star\overline{s_h^2})\leq \frac{C h^{1+\eps}}{\delta^8}.
\end{equation}
Again, if $\psi$ has a non-degenerate stationary point at $p_0\in M$,
\begin{equation}\label{spest:1b}
\left|\int_M \tilde{F}\e^{-2\i\psi/h}b\wedge\star\overline{b}-\frac{ch}{\delta}\e^{-2\i\psi(p_0)/h}\tilde F(p_0)|b(p_0)|^2\right|\leq \frac{Ch^2}{\delta^7},
\end{equation}
where $C$ depends on the a priori bounds on the potentials.

From the discussion in Section \ref{sec:bdry} and Lemma \ref{l:syst-cauchy}, we also have the estimate
\begin{equation}\label{rhsest}
|\langle(\tilde V_2-\tilde V_1)\tilde U_h^{1},\tilde U_h^{2}\rangle_{L^2(\Sigma(\bar M))}|\leq C\e^{c/h}H,
\end{equation}
where we recall that
\[
H=\frac1{|\log d(\mathcal{C}_1,\mathcal{C}_2)|^\gamma}, \quad\text{for some } 0<\gamma\leq1/2.
\]
Combining either \eqref{remest:1a}-\eqref{spest:1a} or \eqref{remest:1b}-\eqref{spest:1b} with \eqref{rhsest}, we get
\begin{eqnarray*}
|Ch\e^{2\i\psi(p_0)/h}\tilde Q(p_0)|a(p_0)|^2|&\leq\frac{ C}{\delta^7}(h^{1+\eps}+\e^{c/h}H),\\
|Dh\e^{-2\i\psi(p_0)/h}\tilde F(p_0)|b(p_0)|^2|&\leq \frac{C}{\delta^7}(h^{1+\eps}+\e^{c/h}H),
\end{eqnarray*}
Equivalently, 
\begin{equation}\label{st-est-aux}
\max\{|\tilde Q(p_0)|,|\tilde F(p_0)|\}\leq \frac{C}{\delta^7}(h^{\eps}+\e^{c/h}H).
\end{equation}
Now, since we assume that $H>0$ is very small, we can choose
\[
h=c((1-\eps)\log(1/H))^{-1}
\]
and get from \eqref{st-est-aux},
\begin{equation}\label{st-est-I}
\max\{|\tilde Q(p_0)|,|\tilde F(p_0)|\}\leq\frac{C}{\delta^7(\log(1/H))^{\eps}}
\end{equation}
for some large enough $C=C(K,M,p,\eps)$.

To justify \eqref{remest:1b} for the terms that are not immediately obvious we use argumets similar to those in \cite{gt-gafa}, e.g.\
\begin{multline*}
\int_M \tilde{F}\e^{-2\i\psi/h}b\wedge\star\overline{s_h^2}=\int_M \tilde{F}\e^{-2\i\psi/h}b\wedge\star\overline{-\dbar_\psi^{*-1}(\tilde Q_2 r_h^2)}\\
=-\int_M \dbar_\psi^{-1}(\tilde{F}\e^{-2\i\psi/h}b)\wedge\star\overline{\tilde Q_2 r_h^2}
\end{multline*}
and use Lemma \ref{l:dbpi} and Proposition \ref{p:diagsols}. Similarly one shows \eqref{remest:1a}.
\end{proof}

Next we will prove an $L^2$-estimate where we take care of those exceptional points that are not included in Lemma \ref{l:mainpf1}.
\begin{lem}\label{l:mainpf2}
There is an $\eps>0$ such that
\[
\||F_1|-|F_2|\|_{L^2(M)}\leq\frac{C}{|\log H|^{2\eps/15}},\text{ where } H=\frac1{|\log d(\mathcal{C}_1,\mathcal{C}_2)|^\gamma},
\]
$C=C(K,M,p,\gamma,\eps)$ and $0<\gamma\leq1/2$.
\end{lem}

\begin{proof}
By Lemma \ref{l:mainpf1}
\begin{equation}\label{pw-est1}
|\tilde F(p)|=||F_1(p)|^2-|F_2(p)|^2|\leq\frac{C}{\delta^7|\log H|^\eps},\quad p\in M\setminus N_{\sqrt\delta}.
\end{equation}
It follows from this and the priori assumptions \eqref{main-apriori1} that there is a $C=C(K,M,p,\gamma,\eps)$ such that
\begin{multline*}
\||F_1|^2-|F_2|^2\|_{L^2(M)}^2=\int_{M\setminus N_{\sqrt{\delta}}}||F_1|^2-|F_2|^2|^2\,\d V_g+\int_{N_{\sqrt{\delta}}}||F_1|^2-|F_2|^2|^2\,\d V_g\\
\leq V_g(M)\left(\frac{C}{\delta^7|\log H|^\alpha}\right)^2+C\delta^2\leq C(\delta^{-14}|\log H|^{-2\eps}+\delta).
\end{multline*}
Choosing 
$\delta=|\log H|^{-2\eps/15}$, we get
\begin{equation}\label{l2-est1}
\||F_1|^2-|F_2|^2\|_{L^2(M)}\leq C|\log H|^{-2\eps/15}.
\end{equation}

Using that
\[
|F_1|-|F_2|=\frac{|F_1|^2-|F_2|^2}{|F_1|+|F_2|}
\]
together with the fact that $|F_1|+|F_2|$ is uniformly bounded from below we can then also to conclude that
\begin{equation}\label{l2-est2}
\||F_1|-|F_2|\|_{L^2(M)}\leq C|\log H|^{-2\eps/15}.
\end{equation}
\end{proof}
 
In the last steps of the proof we indicate how to also get higher Sobolev regularity estimates.
\begin{proof}[Proof of Theorem 1.1/1.2]
Since we assume that $X_j\in T^*M, j=1,2$ are real-valued, we can decompose
$X_j=A_j+\overline{A_j}$, and then
\[
\d X_j = \d A_j+\d\overline{A_j}=\d\pi_{0,1}A_j+\d\pi_{1,0}\overline{A_j}=\del A_j+\dbar\overline{A_j}.
\]
Let now $\alpha_j$ be the primitive functions we introduced earlier, that is in the sense that $\dbar\alpha_j=A_j$ and $F_j=\e^{\i\alpha_j}$. Then
\begin{eqnarray*}
\del F_j=\i\del\alpha_jF_j,&\quad\dbar F_j=\i A_j F_j,\\
\del \overline{F_j}=-\i\overline{A_j}\overline{F_j},&\quad\dbar\overline{F_j}=-\i\dbar\overline{\alpha_j}\overline{F_j}.
\end{eqnarray*}
Observe that $-\del\dbar=\dbar\del$ on functions so $\del\dbar\alpha_j=\del A_j$ and $-\del\dbar\overline{\alpha_j}=\dbar\del\overline{\alpha_j}=\dbar\overline{A_j}$. Consider now
\begin{multline*}
\dbar[|F_j|^2]
=\i|F_j|^2(A_j-\dbar\overline{\alpha_j}),\quad
\del\dbar[|F_j|^2]
=\i|F_j|^2\d X_j-4\,\del|F_j|\wedge\star\dbar|F_j|.
\end{multline*}
This implies that
\begin{multline*}
\del\dbar(|F_1|^2-|F_2|^2)
=\i|F_1|^2\,\d(X_1-X_2)+\i(|F_1|^2-|F_2|^2)\,\d X_2\\
-4\left(\del(|F_1|-|F_2|)\wedge\star\dbar|F_2|+\del|F_1|\wedge\star\dbar(|F_1|-|F_2|)\right).
\end{multline*}
The above is equivalent with
\begin{multline*}
\i|F_1|^2\d(X_1-X_2)=\del\dbar(|F_1|^2-|F_2|^2)-\i(|F_1|^2-|F_2|^2)\,\d X_2\\+4\left(\del(|F_1|-|F_2|)\wedge\star\dbar|F_2|+\del|F_1|\wedge\star\dbar(|F_1|-|F_2|)\right),
\end{multline*}
or rewritten,
\begin{multline}\label{dXid}
\d(X_1-X_2)=-\i|F_1|^{-2}\del\dbar(|F_1|^2-|F_2|^2)-|F_1|^{-2}(|F_1|^2-|F_2|^2)\,\d X_2\\-4\i|F_1|^{-2}\left(\del(|F_1|-|F_2|)\wedge\star\dbar|F_2|+\del|F_1|\wedge\star\dbar(|F_1|-|F_2|)\right).
\end{multline}
Next, under the assumptions that $\|F_j\|_{H^{2+\eta}(M)}\leq K$ for some $\eta>0$, then by Sobolev interpolation,
\begin{eqnarray}\label{hs-est3}
\||F_1|-|F_2|\|_{H^s(M)}&\leq C|\log H|^{-\eps'},\label{hs-est3a}\\
\||F_1|^2-|F_2|^2\|_{H^s(M)}&\leq C|\log H|^{-\eps'},\label{hs-est3b}
\end{eqnarray}
for some 
$0<\eps'<2(1-\alpha)\eps/15, \alpha>0$ and $0<s<2+\eta$. Then from \eqref{dXid} we can conclude that there is an $s'\geq0$ so that
\begin{equation}\label{hs-est4}
\|\d(X_1-X_2)\|_{H^{s'}(M)}\leq C|\log H|^{-\eps'}.
\end{equation}
Similarly we want to give an estimate for the potentials $q_j$, starting from
\begin{equation}\label{l2-est5}
\|\tilde Q\|_{L^2(M)}=\|Q_2/(2|F_2|^2)-Q_1/(2|F_1|^2)\|_{L^2(M)}\leq C|\log H|^{-2\eps/15},
\end{equation}
where $Q_j:=-\star\d X_j+q_j, j=1,2$. Some simple algebra and using that 
the $|F_j|$ are bounded away from zero one can see that
\[
\|Q_1-Q_2\|_{L^2(M)}=\|q_1-q_2-\star\d(X_1-X_2)\|_{L^2(M)}\leq C|\log H|^{-2\eps/15}.
\]
from which it simply follows, similarly as above, that
\begin{equation}\label{hs-est6}
\|q_1-q_2\|_{H^{s'}(M)}\leq C|\log H|^{-\eps'}.
\end{equation}
For $s'=0$ we get Theorem 1.1, and $\eta>1$ would allow for larger $s'>0$.
\end{proof}

%% file: holonomy.tex
\section{Stability for the holonomy}\label{holonomy}

We finally focus our attention towards the holonomy of the 1-form $X:=X_1-X_2$. We can for every closed loop $\gamma$ based at any $m\in M$, consider parallel transport for the connection $\nabla^X=\d+\i X$ on the bundle $M\times\C$. This defines an isomorphism $P_\gamma^X:\C\to\C$, so we may view $P_\gamma^X$ as a non-zero complex number and define the holonomy group of $\nabla^X$ at $m$ by
\[
\text{Hol}(\nabla^X,m):=\{P_\gamma^X\in\C\setminus\{0\};\gamma \text{ is a closed loop based at } m\}.
\]
For real $X$, $P_\gamma^X$ is in fact unitary and it can be observed that $P_\gamma^X=\e^{\i\int_\gamma X}$. When the curvature $\d X=0$ there is a natural group morphism
\[
\rho_m^X:\pi_1(M,m)\to\text{Hol}(\nabla^X,m),
\]
where $\pi_1(M,m)$ is the first fundamental group, consisting of equivalence classes of closed curves up to homotopy. The morphism $\rho^X$ is called the holonomy representation into $GL(\C)$ and it is trivial if and only if $P_\gamma^X=1$ for all closed loops based at $m$, independent of $m$. We will show Theorem \ref{mainthm3}, or more precisely:
\[
\inf_{k\in\Z}\left|\int_\gamma X-2\pi k\right|\leq|\gamma|\omega(d(\mathcal{C}_1,\mathcal{C}_2)),
\]
where $\omega=\omega(x)=C(\log|\log x|)^{-\eps}\to0$ as $x\to0, C=C(K,Mp,\eps)$. Thus the holonomy representation is in this sense close to being trivial. 
\begin{proof}[Proof ot Theorem \ref{mainthm3}]
Let us again consider the functions $F_j$ and define the new function
\[
\Theta:=F_1F_2^{-1}.
\]
Then $\dbar\Theta=\i(A_1-A_2)\Theta$ and hence $\del\overline{\Theta}^{-1}=\i(\overline{A_1}-\overline{A_2})\overline{\Theta}^{-1}$. This implies that
\[
\frac{\dbar\Theta}{\Theta}+\frac{\del\overline{\Theta}^{-1}}{\overline{\Theta}^{-1}}=\i(A_1-A_2)+\i(\overline{A_1}-\overline{A_2})=\i(X_1-X_2)=\i X.
\]
By \eqref{hs-est3b} and Sobolev embedding, $||F_1|^2-|F_2|^2|\leq\omega$, so 
\[
|F_1\overline{F_1}-F_2\overline{F_2}|<\omega.
\]
Multiplying the above inequality with $|\overline{F_1}^{-1}F_2^{-1}|$ we equivalently have
\[
|F_1F_2^{-1}-\overline{F_2}\overline{F_1}^{-1}|=|\Theta-\overline{\Theta}^{-1}|<\omega,
\]
since the $F_j^{-1}$ are bounded away from zero. So we may write
\[
\overline{\Theta}^{-1}=\Theta+\theta,\quad |\theta|<\omega.
\]
Thus we have,
\[
\frac{\dbar\Theta}{\Theta}+\frac{\del\overline{\Theta}^{-1}}{\overline{\Theta}^{-1}}=\frac{\dbar\Theta}{\Theta}+\frac{\del(\Theta+\theta)}{\Theta+\theta}=\frac{\dbar\Theta}{\Theta}+\frac{\del\Theta}{\Theta}+\frac{\Theta\del\theta-\theta\del\Theta}{\Theta(\Theta+\theta)}=\frac{\d\Theta}{\Theta}+\tilde\theta=\i X,
\]
where $|\tilde\theta|\lesssim\omega$ by Sobolev embedding ($X_j\in W^{2,p}, p>2$ implies that $F_j\in W^{3,p}\subset C^{1,\alpha}$ for some $\alpha>0$). 

Let now $\gamma$ be any closed curve on $M$, it will be made up of a finite number of non-self-intersecting loops. So without loss of generality we can assume that $\gamma$ is made up of a single simple loops. We will show that for any such loop,
\begin{equation}
\int_\gamma\frac{\d\Theta}{\Theta}\in2\pi\i\Z,\label{holarg1}
\end{equation}
and hence it will follow that
\[
\inf_{k\in\Z}\left|\int_\gamma X-2\pi k\right|\leq\left|\int_\gamma\i X-\int_\gamma\frac{\d\Theta}{\Theta}\right|=\left|\int_\gamma\tilde\theta\right|\leq|\gamma|\omega,
\]
which is what we want to show.

To show \eqref{holarg1}, suppose we remove a small subarc $\gamma_{q,p}$ between two points $q,p\in\gamma$, so that we are left with another arc $\gamma_{p,q}$ consisting of the remaining points on $\gamma$ between $p$ and $q$. Suppose then that we take a simply connected tubular neighborhood $N(\gamma_{p,q})$ around $\gamma_{p,q}$. In this neighborhood,
\[
g(q):=\int_p^q\frac{\d\Theta}{\Theta},
\]
is well-defined (since the 1-form $\frac{\d\Theta}{\Theta}$ is closed) and by the fundamental theorem of calculus,
\[
\d g(q)=\frac{\d\Theta}{\Theta}(q).
\]
Furthermore, $g(p)=0$ so we can try to find a unique solution to this initial value problem. Take a parametrization of $\gamma=\gamma(t)$ so that $\gamma(0)=p, \gamma(1)=q$. Let $f(t):=\Theta(\gamma(t))$, then $f'(t)=\langle\d \Theta(\gamma(t)),\gamma'(t)\rangle=f(t)\langle\d g(\gamma(t)),\gamma'(t)\rangle$. Now $f(0)=\Theta(p)$ and solving the initial value problem for $f$ hence yields
\[
f(t)=\Theta(p)\e^{g(\gamma(t))}.
\]
Then $\Theta(q)=\Theta(p)\e^{g(q)}$. Taking the limit as $q\to p$ (strictly enlarging the neighborhood $N(\gamma_{p,q})$) we thus find
\[
1=\exp\left(\lim_{q\to p}\int_p^q\frac{\d\Theta}{\Theta}\right).
\]
We can then conclude that \eqref{holarg1} must indeed be true and the theorem follows.

\end{proof}